\documentclass[a4paper,12pt]{article}

\usepackage[T1]{fontenc}
\usepackage{mathpazo}
\usepackage[utf8]{inputenc}
\usepackage{lmodern}
\usepackage[top=2cm, bottom=2.5cm, left=2.5cm, right=2cm]{geometry}
\usepackage{amsmath, amssymb}
\usepackage{hyperref}
\usepackage{color}
\usepackage{graphicx}
\usepackage{tikz}
\usepackage{verbatim}
\usepackage{mathrsfs}
\usepackage{cancel}
\usepackage{cases}
\usepackage{amsthm}
\usepackage{enumerate}
\usepackage{array}
\usepackage{longtable}
\usepackage[affil-it]{authblk}

\theoremstyle{plain}
\newtheorem{thm}{Theorem}[section]
\newcounter{thI}

\newtheorem*{theorem-non}{Theorem}
\newtheorem{lem}{Lemma}[section]
\newtheorem*{theorem-nonfr}{Théorème}
\newtheorem{prop}{Proposition}[section]

\newcounter{cntasmp}

\theoremstyle{definition}
\def\beq{\begin{equation}}
\def\eeq{\end{equation}}
\def\e{\varepsilon}
\def\di{\displaystyle}
\def\RR{\mathbb{R}}
\def\CC{\mathbb{C}}

\theoremstyle{remark}
\newtheorem{rmq}{Remark}[section]

\def\ptn(#1)(#2)(#3){\fill[color=#3](#1)circle(#2)}

\begin{document}

  \renewcommand{\proofname}{Proof}
  \renewcommand\thethI{\arabic{thI}} 
\title{\bf  Front propagation directed by a line of fast diffusion: large diffusion and large time asymptotics}
\author{Laurent  {\sc Dietrich}$^{\hbox{a }}$, Jean-Michel {\sc Roquejoffre}$^{\hbox{b }}$\\
\footnotesize{$^{\hbox{a }}$ Department of Mathematical Sciences,
Carnegie Mellon University\\
Pittsburgh, PA 15213, U.S.A. }\\
\footnotesize{$^{\hbox{b }}$ Institut de Math\'ematiques de Toulouse,
Universit\'e Paul Sabatier}\\
\footnotesize{118 route de Narbonne, F-31062 Toulouse Cedex 4, France}}
\maketitle

\begin{abstract}
\noindent  The system under study is a reaction-diffusion equation in a horizontal strip, coupled to a diffusion equation on its upper boundary via an exchange condition of the Robin type. This class of models was introduced by H. Berestycki, L. Rossi and the second author in order to model biological invasions directed by lines of fast diffusion. They proved, in particular, that the speed of invasion was enhanced by a fast diffusion on the line, the spreading velocity being asymptotically proportional to the square root of the fast diffusion coefficient. These results could be reduced, in the logistic case, to explicit algebraic computations. The goal of this paper is to prove that the same phenomenon holds, with a different type of nonlinearity, which precludes explicit computations. We discover a new transition phenomenon, that we explain in detail.  
\end{abstract}

  \maketitle


\section{Introduction, statement of the problem}
\subsection{Model and question}
Let $\Omega_L$ be the strip $\{(x,y)\in\mathbb R \times (-L,0)\}$. The goal of this paper is to study the large time asymptotics of the following system:
\begin{equation}
\label{e1.1}
\begin{array}{rll}
u_t-Du_{xx}+\mu u-v(t,x,0)=&0\  \  \  (t>0,x\in\mathbb R)\\
v_t-d\Delta v=&f(v)\  \  \  (t>0,(x,y)\in\Omega_L)\\
\ \\
dv_y(t,x,0)+v(t,x,0)=&\mu u(t,x)\  \  \  (t>0,x\in\mathbb R)\\
v_y(t,x,-L)=&0\  \  \  (t>0,x\in\mathbb R)\\
\end{array}
\end{equation}
The unknowns are the functions $(u(t,x),v(t,x,y))$, respectively defined on $\mathbb R_+\times\mathbb R$ and $\mathbb R_+\times\Omega_L$. The positive numbers $\mu$, $d$, $D$ are given. The function $f(v)$ is smooth, and there is $\theta>0$ such that $f\equiv0$ on $[0,\theta]$ and $f(1)=0$. Moreover $f>0$ on $(\theta,1)$ and $f'(1)<0$. Such a nonlinear term will sometimes be referred to as {\it ignition type nonlinearity,} in reference to the mathematical literature on flame propagation models. Of particular interest to us will be the large time asymptotics of \eqref{e1.1}, combined with the limit $D\to+\infty$. 
\subsection{Motivation}
System \eqref{e1.1} was proposed for the first time by Berestycki, Rossi and the second author in \cite{BRR}, as a model for biological invasions in oriented habitats. It was indeed observed in several instances that transportation networks tend to enhance the speed of invasion. Let us mention two biological instances:  the pine processionary moves northwards faster than anticipated, and it is believed that the road network has a responsibility in the phenomenon, see for example \cite{collectif}. The   yellow-legged hornet has invaded the whole   South West of France, as is reported in the maps provided in \cite{patrimoine}: it first followed the main rivers, and from then colonised the inland areas. 

In \cite{BRR}, $\Omega_L$ is replaced by the whole upper half-plane, and the nonlinearity $f$ is a Fisher-KPP type nonlinearity ($f(0)=f(1)=0$, $f>0$ concave between 0 and 1). The line $\{y=0\}$ is named 'the road', and the upper half-plane is named 'the field'. This terminology will be freely used here. We showed the dramatic effect of the road on the overall propagation: there is $c_*(D)>0$ such that, for all $c<c_*(D)$ we have
$$
\lim_{t\to+\infty}\inf_{\vert x\vert\leq ct}u(t,x)=1/\mu,\  \  \lim_{t\to+\infty}\inf_{\vert x\vert\leq ct}v(t,x,y)=1,\ \hbox{locally uniformly in $y\in\mathbb{R}$},
$$
and, for all $c>c_*(D)$ we have
$$
\lim_{t\to+\infty}\sup_{\vert x\vert\geq ct}u(t,x)=0,\  \  \lim_{t\to+\infty}\sup_{\vert x\vert\geq ct}v(t,x,y)=0,\ \hbox{uniformly in $y\in\mathbb{R}$}.
$$
Moreover, there is $c_\infty>0$ such that  $c_*(D)\sim c_\infty\sqrt D$, as $D\to+\infty$. This is in sharp contrast with the classical propagation results for reaction-diffusion equations, such as Aronson-Weinberger \cite{AW}. One could question whether it is an effect of  the Fisher-KPP nonlinearity, or if it holds for more general terms $f$. In \cite{D2}, the first author gives a first hint of the robustness of this phenomenon, by constructing travelling waves
$(\phi(x+ct), \psi(x+ct,y))$
to \eqref{e1.1} whose speed $c$ satisfies indeed $c(D)\sim c_\infty\sqrt D$, where $c_\infty>0$ is characterised in terms of a limiting problem obtained by rescaling $x$ by $\sqrt D$ and sending $D$ to infinity. In order to confirm the phenomenon for \eqref{e1.1} with the ignition type nonlinearity, one should understand  whether, and how,  those travelling waves attract  the solutions of \eqref{e1.1}.  Instead of presenting the results now, we will show some numerical simulations, which reveal a phenomenon that we had not expected.

\subsection{Some numerical simulations}

These simulations were produced using FreeFem++. We used P2 finite elements on a mesh of $400\times 50$ points. The time scheme used is a two-step (to handle the coupling) explicit Euler, which seems quite sufficient in terms of accuracy and speed for our context.  Neumann boundary conditions are imposed on the sides of a domain of size $A\times L$ with $A \gg L$.  Finally, we represented $u$ as a function over the whole domain so that it is visible. The following parameters were used~:

\begin{center}
   \begin{tabular}{ | c | c | }
     \hline
     $\mu u_0,v_0$ & $\mathbf 1_{(-3,3)}(x)$ \\ \hline
     $d$ & 0.1 \\ \hline
     $D$ & 100 \\ \hline
     $\mu$ & 1.4 \\ \hline
	 $\theta$ & 0.3 \\ \hline
    $f(v)$ & $10\times \mathbf 1_{v>\theta}(v-\theta)^2(1-v)$ \\ \hline
      $A$ & 500 \\ \hline
      $L$ & 50 \\ \hline
      $\Delta t$ & 0.1 \\ \hline
   \end{tabular}
 \end{center}
 \begin{center}
\begin{figure}[h!]
  \includegraphics[width=.5\textwidth]{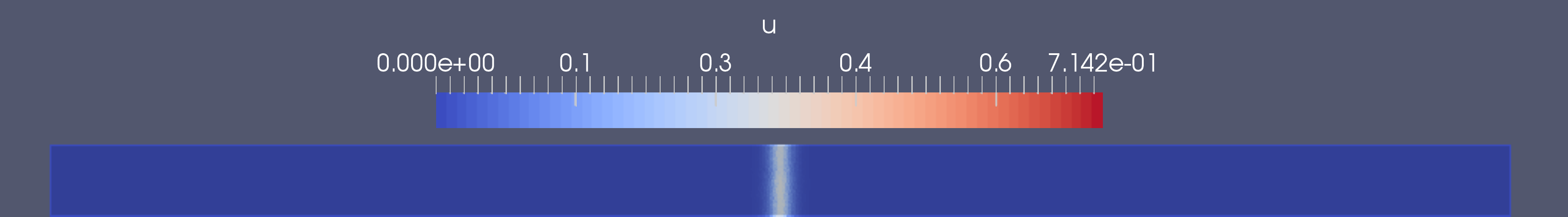}  \includegraphics[width=.5\textwidth]{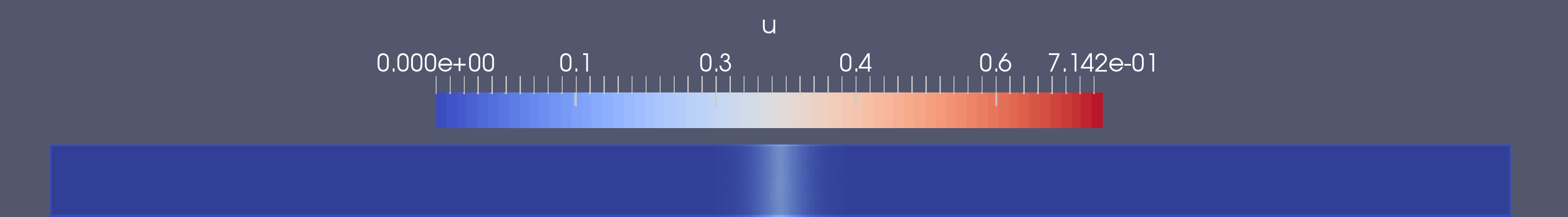}\\
   \includegraphics[width=.5\textwidth]{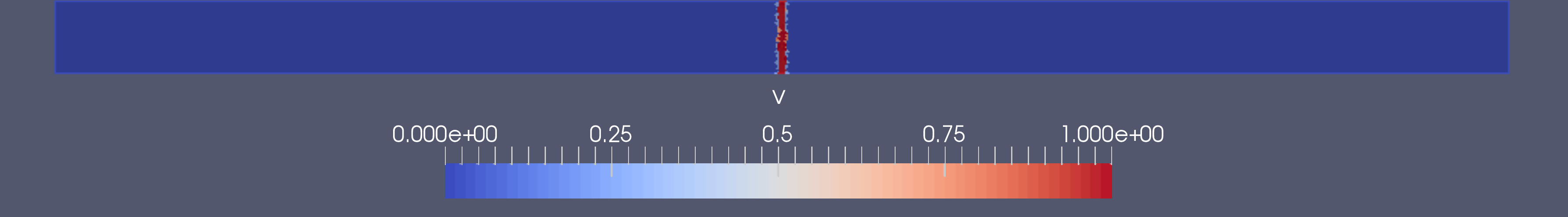}  \includegraphics[width=.5\textwidth]{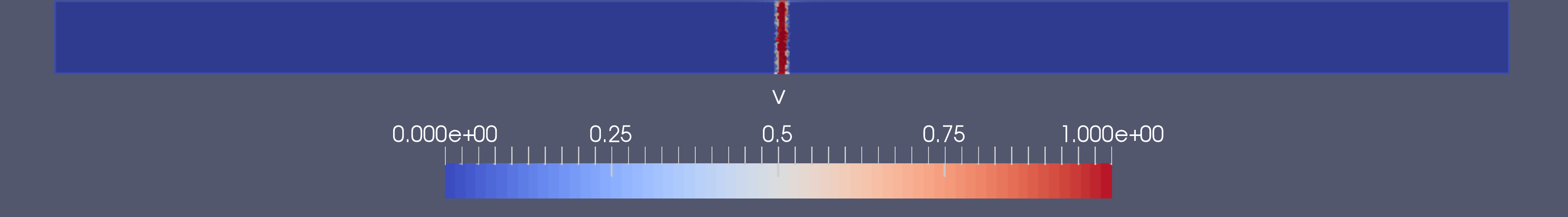} 
\caption{$t=0$ and $t=10\Delta t$}\label{simust0}
\end{figure}


\begin{figure}[h!]
  \includegraphics[width=.5\textwidth]{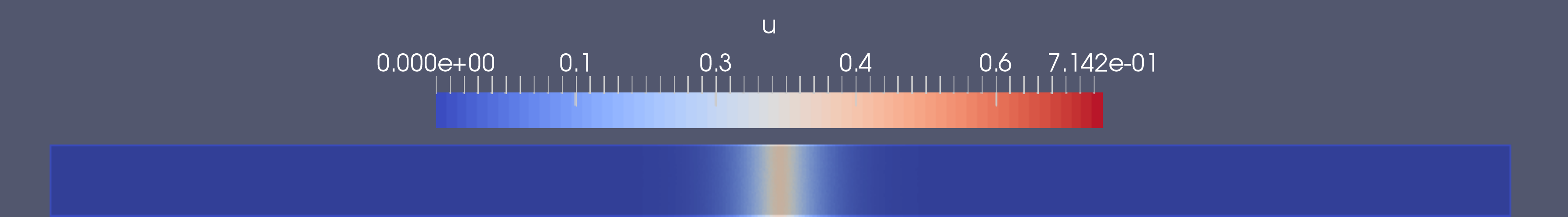}   \includegraphics[width=.5\textwidth]{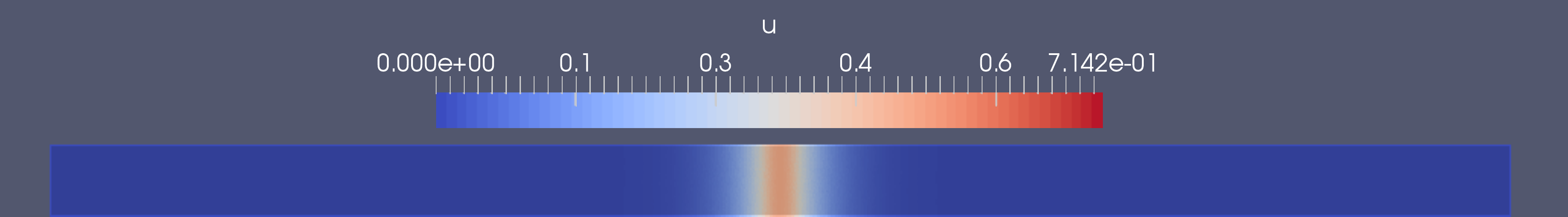}\\
   \includegraphics[width=.5\textwidth]{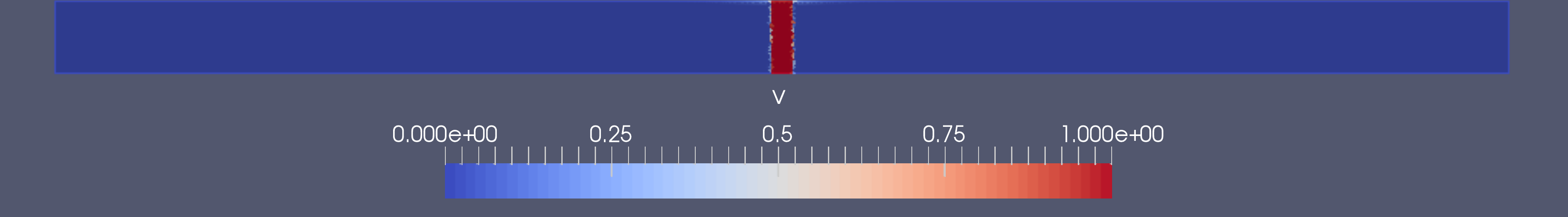}  \includegraphics[width=.5\textwidth]{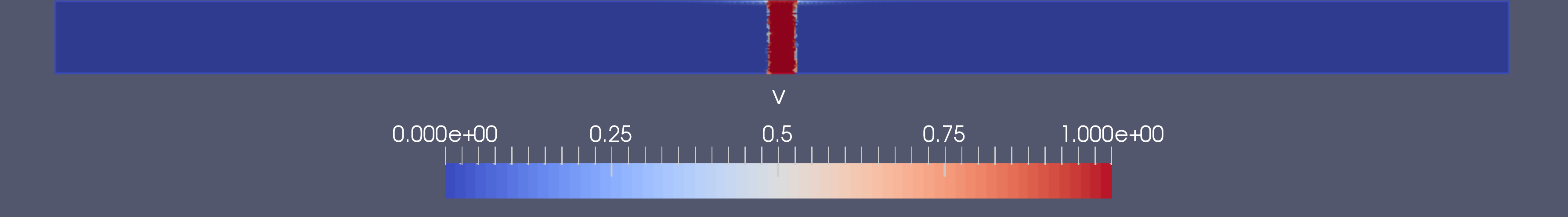} 
\caption{$t=75 \Delta t$ and $t=100\Delta t$ }\label{simust75}
\end{figure}


\begin{figure}[h!]
  \includegraphics[width=.5\textwidth]{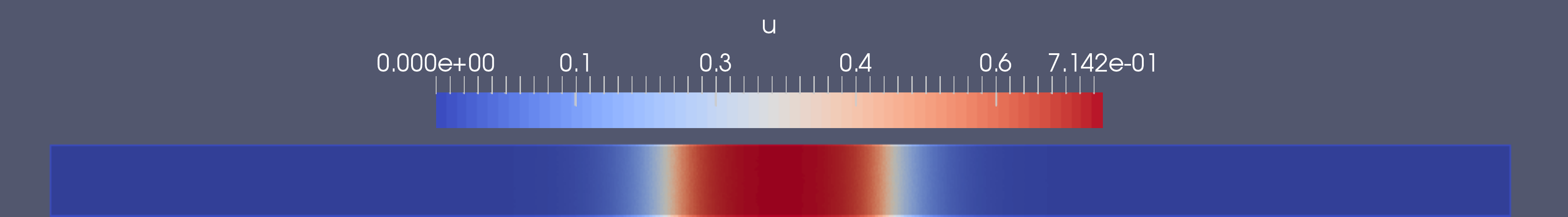}   \includegraphics[width=.5\textwidth]{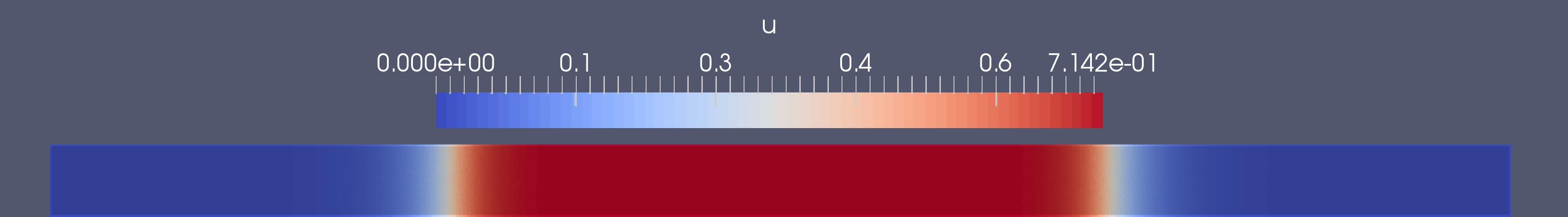} \\
   \includegraphics[width=.5\textwidth]{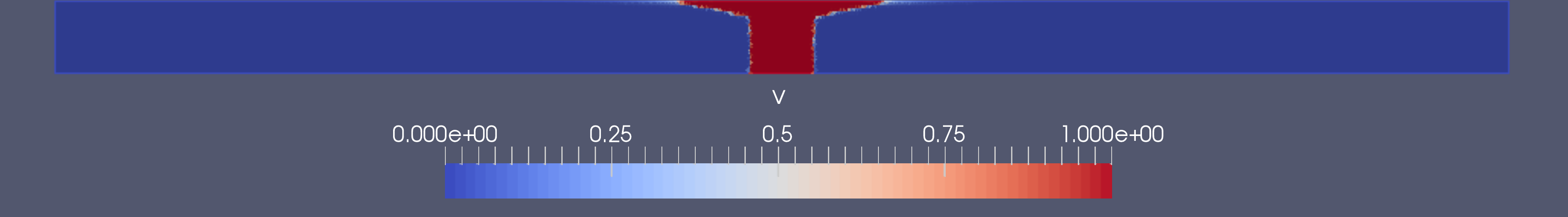}    \includegraphics[width=.5\textwidth]{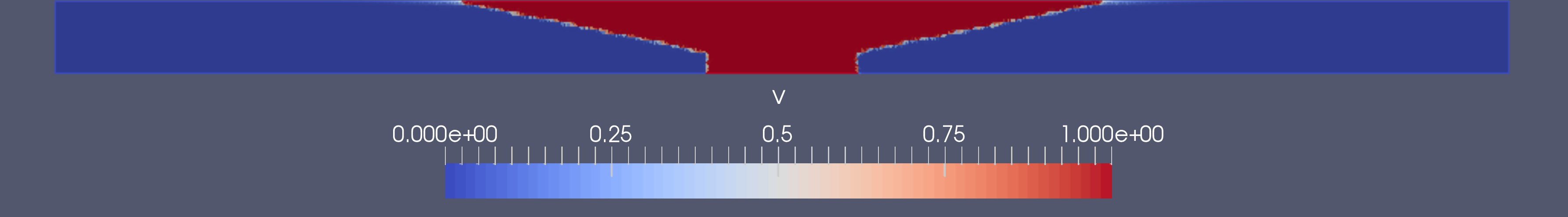} 
\caption{$t=300 \Delta t$ and $t=750\Delta t$}\label{simust300}
\end{figure}

\begin{figure}[h!]
  \includegraphics[width=.5\textwidth]{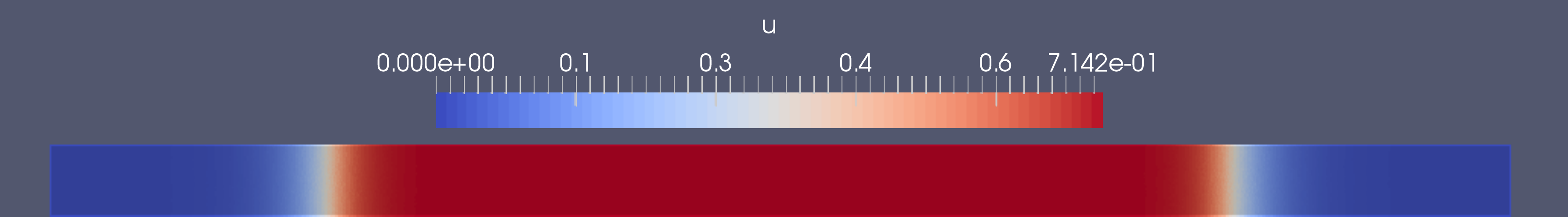} \includegraphics[width=.5\textwidth]{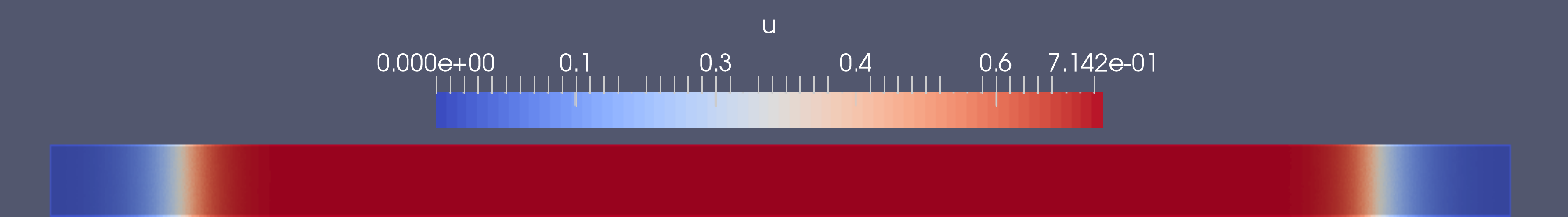}\\
   \includegraphics[width=.5\textwidth]{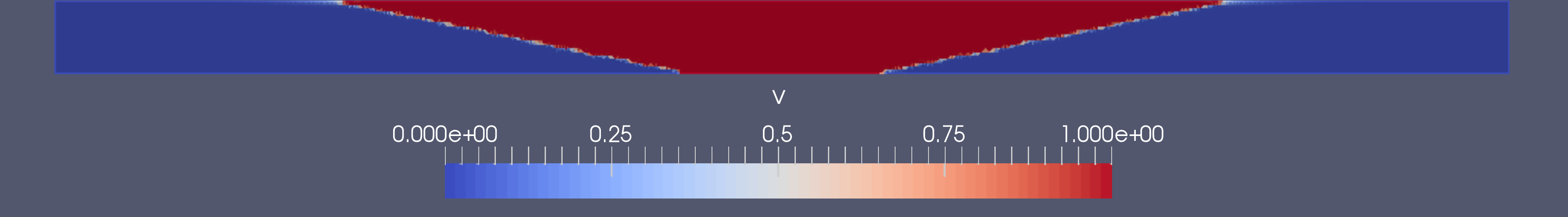}  \includegraphics[width=.5\textwidth]{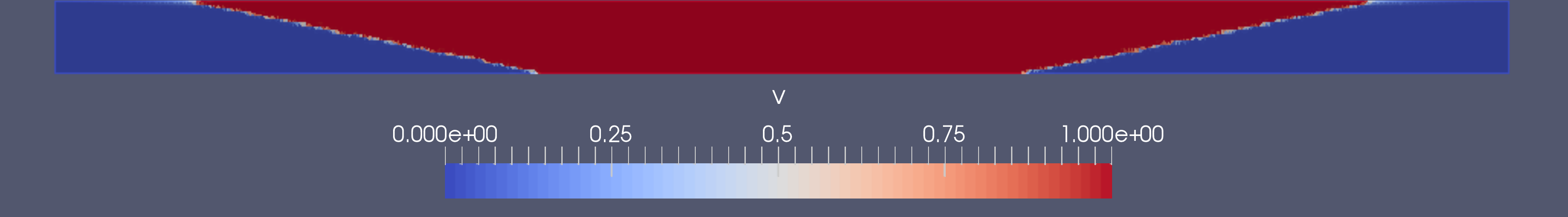} 
\caption{$t=1000 \Delta t$ and $t=1300\Delta t$}\label{simust1000}
\end{figure}
\end{center}

%
\vskip-.6cm
The  scenario that we would expect is thus the following: due to the large diffusivity $D$, $u$ is quickly spread on all $\mathbb R$ and decays rapidly. Meanwhile, $v$ grows slowly and transmits mass to $u$. At some point, $u$ has recovered enough mass and starts to lead the propagation.  The acceleration of the propagation is then transmitted downwards from the road to the bottom of the field, reaching the regime dictated by the travelling wave. The remainder of this paper is devoted to proving that this is indeed what happens.
\section{Main results, discussion}
Let us reformulate System \eqref{e1.1} in the following way, we hope that it will help the reader visualise the problem.
\begin{equation}
\label{normal}
  \begin{tikzpicture}
  \draw (-6,0) -- (6,0) node[pos=0.5,below] {\small{$d\partial_y v = \mu u - v$}} node[pos=0.5,above] {$\partial_t u - D\partial_{xx}^2 u  = v - \mu u$};

  \node at (0,-1.25) {$ \partial_t v - d\Delta v  = f(v)$};

  \draw (-6,-2.5) -- (6,-2.5) node[pos=0.5,above] {\small{$\partial_y v= 0$}};

  \end{tikzpicture}
\end{equation}
We also want to study the behaviour for large $D$, so the renormalization ($x \leftarrow x \sqrt D$) will often be used:
\begin{equation}
\label{readi}
  \begin{tikzpicture}
  \draw (-6,0) -- (6,0) node[pos=0.5,below] {\small{$d\partial_y v = \mu u - v$}} node[pos=0.5,above] {$\partial_t u - \partial_{xx}^2u  = v - \mu u$};

  \node at (0,-1.25) {$\partial_t v - \frac{d}{D} \partial_{xx}^2 v - d\partial_{yy}^2v  = f(v)$};

  \draw (-6,-2.5) -- (6,-2.5) node[pos=0.5,above] {\small{$\partial_y v= 0$}};


  \end{tikzpicture}
\end{equation}
Some results will be stated for equation \eqref{normal} and some for \eqref{readi} and the proofs will juggle between the two. 
We briefly mention existence and uniqueness of a solution and refer to \cite{BRR} for the proof (where the strip is replaced by a half-plane but the argument still holds). We wish to emphasise that uniqueness as well as many properties of this system are a consequence of the monotone structure of \eqref{normal} inherited from the maximum principle investigated in \cite{BRR, D1, D2}. The purpose of this paper is to investigate the large-time and large-diffusion asymptotics of this solution.

\begin{thm}
\label{existence}[Stated for equation \eqref{normal}]
Let $(u_0, v_0) \in \mathcal C(\mathbb R)\times \mathcal C(\Omega_L)$, $0 \leq \mu u_0, v_0 \leq 1$. There exists a global solution in the classical sense to \eqref{normal} with initial data $(u_0,v_0)$. This solution is unique in the class of bounded classical solutions and satisfies $0 \leq \mu u(t,x), v(t,x,y) \leq~1$ for all $t\geq 0, x\in \mathbb R, y \in \Omega_L$.
\end{thm}

\subsection{First results}
The following theorems are natural consequences of the stability of front-like initial data using an argument initiated by \cite{FML} and are not so unexpected. They will, nevertheless, be useful for later purposes. A specificity of the present computations is that they should be uniform in the large parameter $D$, this is why they are detailed.
\begin{thm}
\label{thmintropropafl} [Stated for equation \eqref{normal}]
Let $(u_0, v_0)$ be a front-like initial datum for equation \eqref{readi}, that is $(u_0,v_0) \in \mathcal P_{\alpha_0}$ defined in the next section. There exists an exponent $\omega > 0$ that depends on the initial data only through $\alpha_0$ and for all $\varepsilon > 0$ small enough there exist two shifts $\xi_1^\pm \in \mathbb R$ such that
\begin{alignat*}{2}
\phi(x+c\xi_1^- + ct) - C\varepsilon e^{-\omega t} &\leq \mu u(t,x) &&\leq \mu \phi(x+c\xi_1^+ + ct) + C\varepsilon e^{-\omega t} \\
\psi(x+c\xi_1^- +ct)- C\varepsilon e^{-\omega t} &\leq v(t,x,y) &&\leq \psi(x+c\xi_1^+ +ct ) + C\varepsilon e^{-\omega t}
\end{alignat*}
where $C$ is a constant that depends only on $f$, $d$ and $L$. Moreover, $\omega$ does not depend on $D > d$.
\end{thm}

\begin{rmq}
\label{rmqharnack} 
 The previous theorem does not give the convergence towards travelling waves, but it gives a precise spreading velocity. In \cite{MNRR}, this is the starting point of an iterative argument showing a geometric decrease of the distance separating the two shifts with respect to a fixed time step. One could think of adapting the argument of \cite{MNRR} to \eqref{readi}, but  this would not be uniform in $D > d$. We prefer to focus on the really new features of the model.
\end{rmq}
We now turn to what happens for compactly supported initial data (see Theorem \ref{propacompact} for a precised statement):
\begin{thm}
\label{thmintropropacc} [Stated for equation \eqref{normal}]
Let $(u_0,v_0)$ be non-negative smooth compactly supported data. There exist $\delta > 0$ and $M = O(\sqrt D)$ such that if $\mu u_0, v_0 > 1-\delta$ for $x\in (-M,M)$ then $\mu u, v$ stays trapped (up to an exponentially decaying error) between two shifts of a pair of travelling waves evolving in both directions.
\end{thm}
 \subsection{Data with $O(1)$ support and additional effects}
In this section we state the results that account for the above numerical simulations.
\begin{thm}
\label{thmintropropsmall}  [Stated for equation \eqref{normal}]
Let $L$ be large enough (independently of $D$). There exist $M', \delta' > 0$ independent of $D > d$  such that if the initial datum $(u_0,v_0)$ satisfies
$$v_0 > 1-\delta'\text{ for }x\in (-M',M')$$
then the following holds: let $h(D)$ be any infinitely increasing function as $D\to\infty$; then after a time $t_D = D^{1/2}h(D) + O(1)$, the functions $\mu u$ and $v$ satisfy the assumptions of Theorem \ref{thmintropropacc},  in other words:
 $$\mu u, v \geq 1-\delta\  \hbox{for $x\in(-M \sqrt D,M \sqrt D)$.}
 $$
 \end{thm}
 As a consequence, starting from the time $t=t_D$, propagation occurs as described in Theorem \ref{thmintropropacc}.

One could argue that this happens in a much smaller time. The next theorem shows that, even if the solution may not take all the time
$t_D$ to fall into the assumptions of Theorem  \ref{thmintropropacc},  it stills needs a lot of time. To what extent the upper bound in the preceding theorem, and the lower bound in the next theorem, can be reconciled, is a very interesting question that we leave for future work.
\begin{thm}
\label{t2.1}  [Stated for equation \eqref{normal}]
Let  $M', \delta' > 0$ be as in Theorem  \ref{thmintropropsmall}. For every $\kappa>0$, there exists $C_\kappa>0$ such that, if
\begin{equation}
\label{e2.1}
t\leq C_\kappa D^{1/7-\kappa},
\end{equation}
then $(\mu u(t,.),v(t,.,.))$ does not satisfy the assumptions of Theorem \ref{thmintropropacc}. More precisely we have, uniformly in $t$ satisfying \eqref{e2.1}:
$$
\lim_{D\to+\infty}\Vert u(t,.)\Vert_\infty=0,\  \  \  \  \lim_{D\to+\infty}\frac{\bigl\vert\{(x,y)\in\Omega_L:\ v(t,x,y)\geq\theta\}\bigl\vert}{\sqrt D}=0.
$$
 \end{thm}

Finally, we investigate the  situation of an initial datum supported only on the road. The behaviour that we find  does not at all look like what we have just discovered for initial data supported in the field. If $\mu u_0 \leq 1$ has a support of size $\leq C \sqrt D$  there will be extinction.  On the other hand, we also provide conditions on $\mu$ for invasion to happen.

\begin{thm}
\label{thmintroproparoutel} [Stated for equation \eqref{readi}]
Let $v_0 \equiv 0$ and $\mu u_0 = \mathbf 1_{(-a,a)}$ be initial data for \eqref{readi} and $u, v$ the associated solutions. We have the following~:
\begin{itemize}
\item There exists $a_0 > 0$ independent of $D$ such that if $a < a_0$, $\mu u $ and $v$ decay to $0$ uniformly as $t\to +\infty$.
\item If $a = +\infty$ there are thresholds $\mu^\pm$ independent of $D$ such that for $\mu < \mu^-$ invasion occurs and for $\mu > \mu^+$, $\mu u$ and $v$ converge uniformly to  $1/(\mu(L+1/\mu)) \leq \theta$.
\item More generally, provided $\mu < \mu^ -$, there exists $a_1 > 0$ independent of $D$ such that if $a > a_1$, invasion occurs.
\end{itemize}
\end{thm}

\begin{rmq}
It is quite natural that $\mu$ too large leads to extinction: indeed, we normalised $u$ so that $u \leq 1/\mu$ and moreover $\mu$ acts as a death rate in the equation on $u$. Meanwhile, $v$ sees the same initial boundary Robin condition $\equiv 1$ independently of $\mu$.
\end{rmq}

\subsection{Bibliographical study and discussion}
\label{biblioch3}

The general issue of our work is that of speed-up versus quenching.  The first contribution concerning the behaviour of compactly supported initial data in reaction-diffusion equation of ignition (or bistable) type can be found in Kanel' \cite{KanelFinite}. For the one dimensional equation
$$\partial_t v - \partial_{xx}^2 v = f(v)$$
the author shows the existence of two thresholds $0<L_0\leq L_1<+\infty$ such that if $v_0 = \mathbf 1_{(-l,l)}$ with $l < L_0$, $v$ ends up below $\theta$ in finite time (and as a consequence, decays to $0$ uniformly)~: we call this situation \textit{quenching}. On the other hand, if $l > L_1$ it is shown that $v \to_{t\to+\infty} 1$ uniformly on compact sets. Zlato{\v{s}} \cite{ZlatosSharp} showed  $L_0 = L_1$, and more generally Du and Matano \cite{MatanoDu} showed the existence of such thresholds for  general one-parameter families of initial data.

For equations in cylinders and in the presence of a parallel shear flow,
$$\partial_t v + A\alpha(y)\partial_x v - \Delta v = f(v)$$
an important issue  is to understand how a large amplitude flow (i.e. $A>>1$) will enhance spreading.   This has been studied in various papers starting from \cite{ABP}, where it is showed in the case of a Fisher-KPP nonlinearity, a linear speed-up
 $$c_*(A) \underset{A \to +\infty}{\sim} kA.$$ 
In a more general setting, let us quote Constantin-Kiselev-Oberman-Ryzhik \cite{CKR}, who introduce the notion of \textit{bulk burning rate}. For ignition type nonlinearities, the same result holds as proved by Hamel and Zlatoš \cite{HZ} (see \cite{D2} for a comparison of their result with our situation). As for whether propagation or quenching holds, Constantin-Kiselev-Ryzhik \cite{ConsQuenchFlam} and Kiselev-Zlatoš \cite{KisZla} show that the price to pay for propagation (hence, speed-up)  
 also has a linear scaling in $A$: 
$$L_0 \underset{A \to +\infty}{\sim} k_0 A, \quad L_1 \underset{A \to +\infty}{\sim} k_1 A,
$$ provided that the flow is not constant on too large intervals. In other words, one trades a linear speed up of propagation for a linear growth in the critical size of initial data that leads to quenching. 

In the case of cellular flows, the same phenomenon happens but with a scaling in $A^{1/4}$ (up to a logarithmic factor): the speed-up property was proved by Novikov and Ryzhik \cite{NoviRyz} for the KPP case and more recently by Zlatoš \cite{ZlaCell} for combustion type nonlinearities. On the other hand, Fannjiang-Kiselev-Ryzhik \cite{Fan} proved (for flows with small enough cells) that if $L^4 \ln(L) < kA$ -- where $L$ represents the size of the square supporting the initial datum -- quenching happens. See also the numerical simulations of \cite{Vlad}.

A different type of mechanism is studied in Constantin-Roquejoffre-Ryzhik-Vladimora \cite{ConsRoq08} where the authors investigate a system coupling a reaction-diffusion equation and a Burgers equation. They show different quenching results with respect to a gravity parameter, one of them being that quenching happens independently on $l$ when the gravity is large enough.

In the light of this section, Theorem \ref{thmintropropsmall} may come up as a surprise since it shows a speed-up of the propagation ($c = c_\infty \sqrt D$) for free: $D$ does not appear in the threshold size of the initial data $v_0$. The trade-off is the presence of a "two-speed" mechanism: propagation first happens at a small speed that does not depend on $D$, but accelerates towards the full speed $c(D)$. On the other hand, if one tries to initiate the invasion only thanks to $\mu u_0 = \mathbf 1_{(-l,l)}$, Theorem \ref{thmintroproparoutel} shows that quenching happens if $l < a_0 D^{1/2}$ (from the point of view of  \eqref{normal}).

\subsection{Organisation of the paper}
Section \ref{flsection} is devoted to proving Theorem \ref{thmintropropafl} in the more precise form Theorem \ref{propafrontlike}, Section \ref{sectioncc} provides the details for the proof of Theorem \ref{thmintropropacc} by proving the detailed Theorem \ref{propacompact}. In section \ref{sectionsmallcc} we prove Theorem \ref{thmintropropsmall}, and we prove Theorem \ref{t2.1} in Section \ref{lower}. These two sections will describe more precisely the mechanism that is at work. Finally, the last section investigates the case of initial data supported on the road only.

\section{Front-like initial data}
\label{flsection}
Let us first trap the initial data between functions that will evolve in sub and super-solutions travelling at the right speed. 
\subsection{Trapping the initial data}
\label{sectiontrap}
We deal with bounded, uniformly continuous perturbations $(\rho_1,\rho_2)$ such that there  exist $C$~and~$\alpha_0>0$ such that $\rho_i(x) \leq Ce^{\alpha_0 x} $.
Then we assume
\begin{alignat}{1}
\label{flik1}
&0 \leq \mu u_0, v_0 \leq 1 \\ 
\label{flik2}
&(u_0,v_0) = (\phi(x+\xi),\psi(x+\xi)) + (\rho_1,\rho_2)
\end{alignat}
for some $(\rho_1,\rho_2)$ of the above form and some translation $\xi\in\mathbb R$. Such initial data is said to be in the class $\mathcal P_{\alpha_0}$. In this subsection, we prove that such initial data can be trapped between two translates of the travelling front, which is conceptually simple but necessary.  Due to the degeneracy of $f(v)$ as $v\leq \theta$, we will have to use the following weight function. Let $L_0 > 3$ and  
$$0 < \alpha < \min(\alpha_0, c).
$$ 
Define $\Gamma(x)$ to be a smooth non-decreasing function   such that 
\begin{equation}
\label{defgamma}
\Gamma(x) = \begin{cases}
1\text{ if }x>L_0 \\
e^{\alpha(x+L_0)}\text{ if }x<-L_0-1
\end{cases}\end{equation}
We also recall the exponential convergence towards $0$ or $1$ as $x\to\pm\infty$ proved in \cite{D1}, \cite{D2}: there exist $\lambda, \tilde\lambda > 0$ (bounded from below uniformly in $D>d$) and one can enlarge $L_0 > 0$ so that 
\begin{alignat}{3}
\nonumber &\forall x < -L_0/2 \qquad \mu \phi, \psi &&\leq \frac{\theta}{2} e^{\lambda\left(x+L_0/2\right)} &&\leq \frac{\theta}{2} \\ 
&\forall x > L_0/2 \qquad 1-\mu\phi, 1-\psi &&\leq \frac{1-\theta_1}{2} e^{-\tilde\lambda\left(x-L_0/2\right)} &&\leq \frac{1-\theta_1}{2} \label{expdecaych3}
\end{alignat}
where $\theta < \theta_1 < 1$ is chosen so that $-f'(s) \geq -f'(1)/2 =:\beta > 0$ when $s > \theta_1$.  That way, ahead of the front the system becomes linear and behind the front one controls the monotonicity of $f$. We now  assert the following:

\begin{prop}
\label{proptrapini}
Assume \eqref{flik1},\eqref{flik2}. Then for any $\varepsilon > 0$, there exist $\xi_0^- < 0$ and $\xi_0^+ > 0$ large enough such that 
\begin{alignat}{3}
\mu \phi(x+\xi_0^-) - \varepsilon\Gamma(x+\xi_0^-) &\leq &\mu u_0(x) &\leq \mu \phi(x+\xi_0^+) + \varepsilon \Gamma (x+\xi_0^+) \label{initrap}
 \\
 \psi(x+\xi_0^-,y) - \varepsilon\Gamma(x+\xi_0^-) &\leq &v_0(x,y) &\leq \psi(x+\xi_0^+,y) + \varepsilon \Gamma (x+\xi_0^+) \label{initrap2}
\end{alignat}

\end{prop}

\begin{proof}
We only prove \eqref{initrap}. \eqref{initrap2} is obtained simultaneously with the same arguments ($y$-uniform limits, $y$-uniform exponential decay) by taking $|\xi_0^\pm|$ large enough. We start with the right inequality. Let $\varepsilon > 0$. Thanks to the uniform limit of $\phi$ as $x\to + \infty$, there exists $B_\varepsilon$ independent of $\xi_0^+$ such that for $x\geq -\xi_0^+ + L_0 + B_\varepsilon$, 
$$\mu \phi(x+\xi_0^+) + \varepsilon \Gamma(x+\xi_0^+) \geq 1-\varepsilon + \varepsilon = 1 \geq \mu  u_0(x).
$$
On the other hand, when $x \leq -\xi_0^+ - L_0 - 1$, $\mu  u_0(x) \leq Ce^{\alpha_0 x}$ so here for the inequality to be true, one needs
$\varepsilon e^{\alpha(x+\xi_0^+ + L_0)} \geq Ce^{\alpha_0 x} .$ 
But since $x+\xi_0^+ + L_0 < 0$ and $0 < \alpha < \alpha_0$, one just needs 
$\varepsilon e^{\alpha_0(x+\xi_0^+ + L_0)} \geq Ce^{\alpha_0 x} ,$ 
which is ensured as soon as $\xi_0^+ > \ln(C/\varepsilon) - L_0$.

Now only the compact region $x\in (-\xi_0^+ - L_0 - 1, -\xi_0^+ + L_0 + B_\varepsilon)$ remains. Observe that on this interval, $\mu u_0(x)$ goes uniformly to $0$ as $\xi_0^+ \to \infty$, whereas the right-hand side in \eqref{initrap} has a fixed positive infimum, so that the desired order is obtained by enlarging $\xi_0^+$.

For the existence of $\xi_0^-$: observe that on $x\geq -\xi_0^- + L_0 +1$
$$\mu  \phi(x+\xi_0^-) - \varepsilon \Gamma(x+\xi_0^-) \leq 1 - \varepsilon \leq \mu u_0(x)$$
provided $\xi_0^-$ is negative enough, thanks to the uniform limit of $u_0$ as $x\to +\infty$.
Now for the rest of the proof, we need on $x\leq -\xi_0^-  + L_0 + 1$
$$\varepsilon e^{\alpha(x+\xi_0^- + L_0)} \geq \mu \phi(x+\xi_0^-) - (\mu \phi(x) + \rho(x))$$
Because the exponential decay $\lambda$ of $\phi$ and $\psi$ satisfies $\lambda > c \geq \alpha$ (see \cite{D2}) this is true on $x\leq -\xi_0^- - L_0 - B_\varepsilon$ with $B_\varepsilon > 0$ large enough independent of $\xi_0^-$ so that here
$$\varepsilon e^{\alpha(x+\xi_0^- + L_0)} \geq \theta e^{\lambda(x+\xi_0^-)} \geq \mu \phi(x+\xi_0^-)$$
Again, we cover the compact region left around the interface by enlarging $-\xi_0^-$.
\end{proof}

\subsection{Wave-like sub and supersolution}
\label{wavelikes}
%
%
%
%
We adapt the original result of Fife-McLeod \cite{FML} using the simplified notations and generalisation of Mellet-Nolen-Ryzhik-Roquejoffre \cite{MNRR}. The adaptation is computationally non trivial, so let us first explain the changes that we expect to happen. Our objective is to build a supersolution $\overline u, \overline v$ to \eqref{readi} that is close to the front $(\phi,\psi)$ (in the frame moving at speed $c$). In the homogeneous case and for generalized transitions fronts the authors of \cite{MNRR} proposed 
$\bar v = \psi(x+c\xi(t)) + q(t)\Gamma(x+c\xi(t))$, where $\psi$ is the front, $\xi(t)$ is an increasing shift starting from the initial one and converging to some finite limit, $q(t) = \varepsilon e^{-\omega t}$ and $\Gamma$ is defined above: this is a necessary correction to take into account the initial perturbation and the degeneracy of $f$ on $v\leq \theta$.
In our case, we will  look for 
\begin{equation}
\label{sursolfrontlike}
\begin{cases}
\mu \bar u = \mu \phi(x+c\xi(t)) + q_u(t)\Gamma(x+c\xi(t)) \\
\bar v = \psi(x+c\xi(t),y) + q_v(t,y)\Gamma(x+c\xi(t))
\end{cases}
\end{equation}
and
\begin{equation}
\label{subsolfrontlike}
\begin{cases}
\mu \underline u = \mu \phi(x-c\xi(t)) - q_u(t)\Gamma(x-c\xi(t)) \\
\underline v = \psi(x-c\xi(t),y) - q_v(t,y)\Gamma(x-c\xi(t))
\end{cases}
\end{equation}
with $\xi$ starting from $\xi_0^\pm$. We will also use the fact that
$$\forall M > 0, \exists \delta_M > 0 \mid \partial_x \phi, \partial_x \psi > \delta_M\text{ when } x\in(-M,M).
$$
Now we reduce $\alpha$ a bit more and set 
$
\alpha = \min(\alpha_0, c/5),$
just so that the quantities 
$$
\alpha c - d/D \alpha^2 \geq \alpha c - \alpha^2 > \alpha c /4 - \alpha^2 > 0
$$
cannot be zero. These quantities will play an important role in the following computations. Observe that this condition on $\alpha$ means that the decay correction obtained through $\Gamma$ is limited: solutions starting from large perturbations (i.e. small $\alpha_0$) will be stabilized thanks to a correction with an $\alpha_0$ decay also, but solutions from very small perturbations (i.e. very large $\alpha_0$) will still need a $c/5$ correction in the decay at $-\infty$ to be stabilized.

Since we still want an exponential decay of $q_v(t)$ we look for 
\begin{equation}
\begin{cases}
\label{quqv}
q_u(t)=\varepsilon C e^{-\omega t} \\ 
q_v(t,y)= \varepsilon h(y)e^{-\omega t} 
\end{cases}
\end{equation}
with separate variables. The boundary conditions yield $h'(-L) = 0$ and $h'(0) + h(0) = C$ so that we have a large choice for $h$. Nonetheless, it will become clear in the following computations that a good candidate is
\begin{equation}
\label{defh}
C = \cosh\left(\sqrt{\kappa/d}L\right) + \sinh\left(\sqrt{\kappa/d}L\right), \quad
h(y) = \cosh\left(\sqrt{\kappa/d}\left(y+L\right)\right) 
\end{equation}
with 
$$
\kappa = \min(\beta/2, (\alpha c - d/D \alpha^2)/2) > 0,\quad
\omega = \min(G(\sqrt{\kappa/d}L), \beta/2, \text{Lip} f, \alpha c/4 - \alpha^2) > 0
$$
and $G(x) = \di\frac{\mu \tanh(x)}{1+\tanh(x)}$. The role of these conditions will be clear in the computations.

Observe that since $G'(0) > 0$, the decay exponent $\omega$ is then linearly small as $\beta$ or $\alpha_0$ or $\mu$ is small, but it should be noticed that it does not depend on $D \geq  d$, and that it depends on the initial data only through $\alpha_0$.
We can now state the following:
\begin{thm}
\label{propafrontlike}
Assume \eqref{flik1},\eqref{flik2} and let $u,v$ denote the associated solutions of \eqref{readi}. Let $\varepsilon_0 = \min(\theta/4, (1-\theta_1)/4, \gamma_0)$ where $$\gamma_0 = \frac{1}{4B},\quad B = \left(\frac{3\text{Lip} f + |\Gamma|_{C^2}}{c\delta_{L_0+2}}\right)C \max(1,1/\mu).$$
There exists a constant $K_0$ that depends on the initial data only through $\alpha_0$ and such that if $\varepsilon \in (0, \varepsilon_0)$, there exists $\xi_1^\pm$ with
\begin{equation}
\xi_1^+ \leq  \xi_0^+ + \varepsilon K_0, \quad \xi_1^- \geq  \xi_0^- - \varepsilon K_0
\end{equation}
and for all $t \geq 0$,
\begin{alignat}{2}
\label{trapu}
\phi(x+c\xi_1^-) - q_u(t)\Gamma(x+c\xi_1^-) &\leq \mu u(t,x-ct) &&\leq \mu \phi(x+c\xi_1^+) + q_u(t)\Gamma(x+c\xi_1^+) \\
\label{trapv}
\psi(x+c\xi_1^-) - q_v(t,y)\Gamma(x+c\xi_1^-) &\leq v(t,x-ct,y) &&\leq \psi(x+c\xi_1^+) + q_v(t,y)\Gamma(x+c\xi_1^+)
\end{alignat}
\end{thm}
\begin{proof}
\begin{figure}[h!]
 \centering \def\svgwidth{500pt} 
 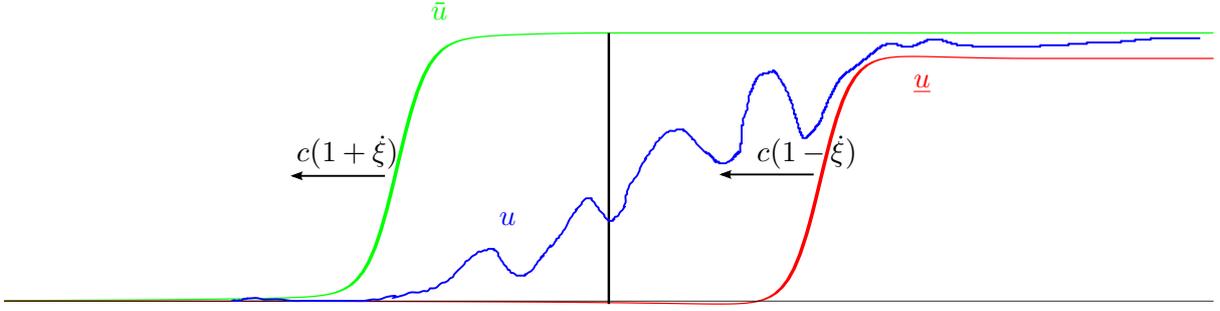 \caption{Trapping of the front-like data} 
 \label{figfrontlike}
\end{figure}

Inequations \eqref{trapu},\eqref{trapv} are set in the moving frame with variables $(t,x+ct)$. As a consequence, in the computations one has to replace $\partial_t$ by $\partial_t + c\partial_x$. We now want to show that $\bar u, \bar v$ as defined in \eqref{sursolfrontlike} yields indeed a supersolution:
$$\mathcal N\begin{pmatrix} \overline u \\ \overline v \end{pmatrix} \geq \begin{pmatrix} 0 \\ 0 \end{pmatrix}$$
where
$$\mathcal N\begin{pmatrix} u \\ v \end{pmatrix} = \begin{pmatrix} u_t - u_{xx} + cu_x + \mu u - v(\cdot,0) \\ v_t - \frac{d}{D} v_{xx} - dv_{yy} + cv_x - f(v)\end{pmatrix}$$
and that $\underline u,  \underline v$ as defined in \eqref{subsolfrontlike} yields a subsolution. Then \eqref{trapu},\eqref{trapv} will follow by an application of the comparison principle, Prop. \ref{proptrapini}, and the monotonicity of $\xi$. Indeed, this will show that in the original frame $u,v$ stays trapped between the fronts shifted initially by $\xi_0^\pm$ and moving at speed resp. $c(1\pm\dot\xi)$ (or speed $c\pm\dot\xi$ in the moving frame). This deformation becomes of course exponentially small over time due to the $e^{-\omega t}$ factor. Observe also that $\dot\xi \leq 1/4$ and is exponentially decaying over time, so $u, v$ will propagate at least and at most with speed $c+o(1)$.

We divide this computation in three zones concerning $x+\xi(t)$. In the following, $\phi$ and $\psi$ will always mean $\phi(x+c\xi(t))$ and $\psi(x+c\xi(t),y)$, $q_u$ will always mean $q_u(t)$, $q_v$ will either mean $q_v(t,0)$ or $q_v(t,y)$ and $\Gamma$ will always mean $\Gamma(x+c\xi(t))$, all of these functions being defined as above in \eqref{defgamma} and \eqref{quqv}-\eqref{defh}.

\subsubsection{Behind the front: $x+\xi(t) > L_0+1$}
Here $\Gamma \equiv 1$ and $\psi, \overline v \geq (1+\theta_1)/2$ so that 
\begin{alignat*}{1}
\mathcal N\begin{pmatrix} \overline u \\ \overline v \end{pmatrix}_{1} &= c\dot \xi \phi_x + \dot q_u/\mu -  \phi_{xx} + c \phi_x + \mu\phi + q_u - \psi - q_v(t,0) \\ 
&= c \dot\xi \phi_x + \dot q_u/\mu + q_u - q_v \\
&\geq \dot q_u/\mu + q_u - q_v \\
&= \varepsilon e^{-\omega t}\left(-Cw/\mu + C - \cosh\left(\sqrt{\kappa/d}L\right)\right) \\
&= \varepsilon e^{-\omega t}\left(-\left(\cosh\left(\sqrt{\kappa/d}L\right)+\sinh\left(\sqrt{\kappa/d}L\right)\right)w/\mu + \sinh\left(\sqrt{\kappa/d}L\right)\right) \geq 0
\end{alignat*}
The first inequality holds because we look for $\dot\xi \geq 0$ and the last because $\omega \leq G(\sqrt{\kappa}L)$. 
\begin{alignat*}{1}
\mathcal N\begin{pmatrix} \overline u \\ \overline v \end{pmatrix}_{2} &= c\dot \xi \psi_x + \dot q_v -  d/D \psi_{xx} + c \psi_x - d\psi_{yy} - d\partial_{yy}^2 q_v - f(\psi) + f(\psi) - f(\overline v) \\ 
&= c\dot \xi \psi_x + \dot q_v + f(\psi) - f(\overline v) - d\partial_{yy}^2 q_v \\
&\geq \dot q_v - d\partial_{yy}^2 q_v + \beta q_v = \varepsilon e^{-\omega t} h(y) (-\omega - \kappa + \beta)\geq \varepsilon e^{-\omega t} h(y) (-\omega + \beta/2) \geq 0.
\end{alignat*}
The last inequality holds because $w\leq \beta/2$ and the next to last because $\kappa \leq \beta/2$.

\subsubsection{Ahead of the front: $x + \xi(t) < -L_0 - 1$}
Heres $\Gamma(x+\xi(t)) = e^{\alpha(x+\xi(t)+L_0)}$, $\psi \leq \theta/2$ and $\bar v \leq \psi + \varepsilon \leq 3\theta/4 \leq \theta$ so $f(\overline v\equiv0$.

\begin{alignat*}{1}
\mathcal N\begin{pmatrix} \overline u \\ \overline v \end{pmatrix}_{1} &= c \dot\xi \phi_x + \left(\frac{\dot q_u}{\mu}  + \frac{q_u}{\mu} \alpha c\dot\xi - \frac{q_u}{\mu} \alpha^ 2 + c \frac{q_u}{\mu} \alpha + q_u  - q_v(\cdot,0) \right)e^{\alpha(x+\xi+L_0)}  \\
&\geq \frac{1}{\mu}\left( \dot q_u + q_u\alpha c \dot\xi - q_u\alpha^2 + c\alpha q_u\right)e^{\alpha(x+\xi+L_0)}  \geq \frac{1}{\mu}\left( -\omega + \alpha c - \alpha^2\right)e^{\alpha(x+\xi+L_0)}q_u \geq 0.
\end{alignat*}
The last inequality holds because $\omega \leq (\alpha c - \alpha^2)/2$, and the first because $q_u(t) \geq q_v(t,0)$.
\begin{alignat*}{1}
\mathcal N\begin{pmatrix} \overline u \\ \overline v \end{pmatrix}_{2} &= c\dot\xi \psi_x + e^{\alpha(x+\xi+L_0)}\left(\dot q_v + q_v\left(\alpha(c\dot \xi + c) - d/D \alpha^2 \right) - d\partial_{yy}^2 q_v\right) \\
&\geq e^{\alpha(x+\xi+L_0)}q_v \left( -\omega + \alpha c\dot \xi + \alpha c - d/D \alpha^2 - \kappa \right) \\
&\geq e^{\alpha(x+\xi+L_0)}q_v \left( -\omega + (\alpha c - d/D \alpha^2)/2 \right) \geq 0
\end{alignat*}
The last inequality holds because of the condition on $\omega$, and the next to last because of the condition on $\kappa$ and because $\dot\xi \geq 0$.

\subsubsection{The middle region: $|x+\xi(t)| < L_0 + 2$}
We have 
\label{middle}
\begin{alignat*}{1}
\mathcal N\begin{pmatrix} \overline u \\ \overline v \end{pmatrix}_{1} &= c \dot\xi \phi_x + \frac{\dot q_u}{\mu} \Gamma + c\dot\xi \frac{q_u}{\mu} \Gamma_x - \frac{q_u}{\mu} \Gamma_{xx} + c \frac{q_u}{\mu} \Gamma_x + (q_u - q_v)\Gamma \\
& \geq c\dot\xi \phi_x + \frac{\dot q_u}{\mu} \Gamma - \frac{q_u}{\mu} \Gamma_{xx}   \geq c\dot\xi \delta_{L_0+2} - (\omega+|\Gamma|_{C^2})\frac{q_u}{\mu} \geq 0,
\end{alignat*}
provided \begin{equation} \label{condxi1}\dot\xi \geq \frac{\omega+|\Gamma|_{C^2}}{c\delta_{L_0+2}} q_u.
\end{equation}
And we have
\begin{alignat*}{1}
\mathcal N\begin{pmatrix} \overline u \\ \overline v \end{pmatrix}_{2} &\geq c\dot\xi\psi_x - q_v \text{Lip} f + \dot q_v \Gamma + c\dot\xi q_v \Gamma_x - d/D q_v \Gamma_{xx} + c q_v\Gamma_x - d\partial_{yy}^2q_v\Gamma  \\
&\geq c\dot\xi \delta_{L_0+2} - q_v \text{Lip} f - \omega q_v - d/D |\Gamma|_{C^2} q_v - \kappa q_v \geq 0,
\end{alignat*}
provided  \begin{equation} \label{condxi2}\dot\xi \geq \frac{\text{Lip} f + \omega + d/D |\Gamma|_{C^2} + \kappa}{c\delta_{L_0+2}} q_v.
\end{equation}

We obtain conditions \eqref{condxi1}, \eqref{condxi2} by remarking that $\kappa \leq \beta < \text{Lip}(f)$, $\omega < \text{Lip}(f)$ and $d/D < 1$ and then we take 
$\dot \xi(t) = B \varepsilon e^{-\omega t},$ so that
$$
\xi(t) = \xi_0^+ + \frac{B\varepsilon(1-e^{-\omega t})}{\omega}, \quad K_0 = B/\omega
$$
answer our queries. One should observe that the condition $B \varepsilon \leq 1/4$ has not been used yet as well as $\omega \leq c\alpha/4 - \alpha^2$ rather than just $1/2(c\alpha - \alpha^2)$.
Observe that the computations concerning the subsolution \eqref{subsolfrontlike} with this time 
$$
\xi(t) = -\xi_0^- + \frac{B\varepsilon(1-e^{-\omega t})}{\omega}
$$
are exactly symmetric, except for a $c\alpha(1-\dot\xi)$ term (instead of $c\alpha(1+\dot\xi)$) that appears ahead of the front and in the middle region, which is treated thanks to the above still unused assumptions:
\begin{alignat*}{1}
-\omega + c\alpha(1-\dot\xi) - d/D \alpha^2 - \kappa &\geq -\omega + \frac{3c\alpha}{4} - d/D \alpha^2 - \kappa  \\ &\geq -\omega + \frac{3c\alpha}{4} - \alpha^2 - \kappa \\ 
&\geq -\omega + \frac{c\alpha}{4} - \frac{\alpha^2}{2} \\ 
&\geq -\omega + \frac{c\alpha}{4} - \alpha^2  \geq 0.
\end{alignat*}
On the other hand, we have
$$
-\omega + c\alpha(1-\dot\xi) - \alpha^2 \geq -\omega + \frac{3c\alpha}{4} - \alpha^2
\geq -\omega + \frac{c\alpha}{4} - \alpha^2 \geq 0.
$$
This ends the construction. \end{proof}

\section{Compactly supported initial data}
\label{sectioncc}
In this section, we go back in the fixed original frame. Seeing the problem in the light of \cite{FML}  it is natural to test: 
$$\begin{pmatrix} \underline u \\ \underline v \end{pmatrix} = \begin{pmatrix} \phi(x+ct+c\xi(t)) + \phi(-x+ct+c\xi(t)) - 1/\mu \\ \psi(x+ct+c\xi(t)) + \psi(-x+ct+c\xi(t)) - 1 \end{pmatrix}
$$
as a subsolution to \eqref{readi}, i.e. a pair of waves evolving in opposite directions. Of course, in light of the previous section, for this to be a subsolution one needs a well chosen correction in time and in space (in the degeneracy regime of $f$). Let us define the symmetrised fronts
$$
\tilde{\phi}(\cdot) = \phi(-\cdot),\quad \tilde{\psi}(\cdot) = \psi(-\cdot).
$$
In the sequel we will always use the following notations:
$$
\phi = \phi(x+ct+\xi_0 - c\xi(t)),\quad \tilde{\phi} = \tilde\phi(x-ct-\xi_0 + c\xi(t)),
$$
and the same will hold for $\psi, \tilde{\psi}, \Gamma, \tilde{\Gamma}$. 
Here $\xi_0$ will be a large initial shift  and $\xi(t)$ a time-increasing shift with $\xi(0) = 0$ and $c\xi(+\infty) \leq 1$,
which will be realised as a smallness condition on $\varepsilon_0$. In this section we set
\begin{equation}
\label{defalpha}
\alpha = \min(\lambda,\tilde\lambda,c/5)
\end{equation}
where $\lambda$ and $\tilde\lambda$ are already defined in \eqref{expdecaych3} so that $\alpha$ yields the same inequations as above and moreover $\alpha < \lambda, \tilde\lambda$.
$\Gamma$ is defined as above, only with a little more margin. Precisely let us set this time:

\begin{equation}
\Gamma(x) = \begin{cases}
1\text{ if }x>L_0-1 \\
e^{\alpha(x+L_0)}\text{ if }x<-L_0+1
\end{cases}\end{equation}

We will set the following:
$$
\underline u = \max\left(0,\phi + \tilde{\phi} - 1/\mu - q_u(t)/\mu\min(\Gamma, \tilde\Gamma)\right),\ 
\underline v = \max\left(0,\psi + \tilde\psi - 1 - q_v(t,y)\min(\Gamma, \tilde\Gamma)\right).
$$
The proof will consist in adapting the previous computations. We shall see that $(\underline u, \underline v)$ yields a subsolution provided only a size condition on the initial shift $\xi_0$ (independently of $D > d$). This condition is important, because then for the initial data to lie above $(\underline u(0), \underline v(0))$ it has to be large enough on a large enough interval. Moreover, we wish to insist on the fact that to retrieve the original model \eqref{normal} one has to change the variable $x \leftarrow x / \sqrt{D}$. As a consequence, when stated for \eqref{normal}, our result assumes that $u_0, v_0$ are large enough on an interval with length of order $\sqrt{D}$. Theorem \ref{thmintropropacc} will be proved as soon as we have proved the 
\begin{thm}
\label{propacompact}
1. There exist $\varepsilon_0 > 0$ small enough and two constants $B, \xi_0 > 0$ large enough such that for all $0<\varepsilon<\varepsilon_0$, there exist a small $\delta > 0$ and $M > 0$ such that if $0 \leq \mu u_0,v_0 \leq 1$ satisfy $\mu u_0,v_0 > 1-\delta$ on $x\in(-M,M)$, then 
$$\underline u = \max(0, \phi+\tilde\phi - 1/\mu - q_u/\mu\min(\Gamma,\tilde\Gamma)),\quad \underline v = \max(0, \psi+\tilde\psi - 1 - q_v\min(\Gamma,\tilde\Gamma))$$
where $q_u = \varepsilon C e^{-\omega t}$ and $q_v = \varepsilon h(y) e^{-\omega t}$ are defined as above and this time 
$\di  \xi(t) = \di\frac{B\varepsilon(1-e^{-\omega t})}{\omega}$  defines a subsolution to \eqref{readi} with initial data $u_0,v_0$ for all times. By the comparison principle, we then have at all times
$$\underline u \leq u, \underline v \leq v.
$$
As a consequence, \eqref{readi} propagates the initial data $u_0,v_0$ along the $x$-axis with speed at least as $c + \underset{t\to+\infty}{o}(1)$ in both directions. 

\noindent 2.Using the notations of Section 2 we have the following: let $\tilde u, \tilde v$ denote the same functions as in \eqref{sursolfrontlike} with $\phi,\psi$ and $\Gamma$ replaced by $\tilde\phi,\tilde\psi,\tilde\Gamma$. As a consequence, $\tilde u, \tilde v$ will be a supersolution for decreasing front-like initial data. Up to enlarging the initial shifts, we assert that 
$$(\min(\bar u,\tilde u), \min(\bar v,\tilde v))$$
is a supersolution to \eqref{readi} with initial data $u_0,v_0$ for all times. Again, this implies that 
$$u \leq \min(\bar u,\tilde u), v \leq \min(\bar v,\tilde v)$$
and so that the level lines of $u,v$ propagate at most as $c + \dot \xi = c + o(1)$ in both directions along the $x$-axis.
\end{thm}

\begin{rmq}
(i). As noticed above, observe that one needs to replace $M \leftarrow M \sqrt D$ when Theorem \ref{propacompact} is stated for the original system \eqref{normal}.

\noindent (ii). The size condition on $u_0,v_0$ is far from optimal and ensures only that $u_0 \geq \underline u(0), v_0 \geq \underline v(0)$. It could be sharpened by replacing $1-\delta$ with $\theta$ and by waiting long enough for the reaction to put $u, v$ above $1-\delta$.
\end{rmq}

\begin{figure}[h!]
 \centering \def\svgwidth{450pt} 
 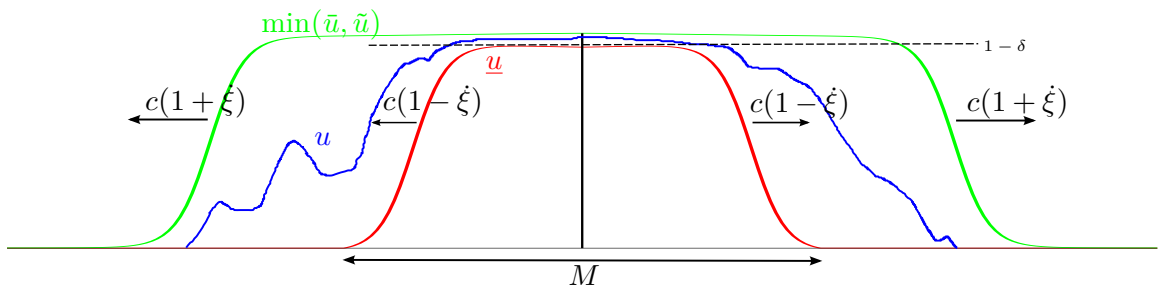 \caption{Trapping of the compactly supported data} 
 \label{figcomp}
\end{figure}

\begin{proof}
The second part of Theorem \ref{propacompact} is easy because the minimum of two supersolutions is a supersolution and any front like initial data can be translated above any compactly supported initial data. 

The first part is more intricate. Observe that $\underline u(0), \underline v(0)$ are zero except on a set of length $(-M,M)$ (with $M$ proportional to $\xi_0$) and that on $(-M,M)$ they are less than some $1-\delta$: this directly gives the largeness condition asked so that $u_0 \geq \underline u(0), v_0 \geq \underline v(0)$. We now detail the computation of $\mathcal N(\underline u, \underline v)$ in the following subsections by splitting the computations in three zones concerning $x+ct+\xi_0$.

\subsection{$x+ct+c\xi_0 < -L_0$}
In this zone, one has necessarily $x+ct+\xi_0 - \xi(t) < -L_0$ and also $x-ct-\xi_0 + \xi(t) < -L_0$ (by asking $2\xi_0 \geq 1$). As a consequence, in this zone we have $\mu\phi, \psi, \underline v \leq \theta/2$ and $\mu\tilde{\phi}, \tilde{\psi} \geq (1+\theta_1)/2$. Also $\min(\Gamma,\tilde{\Gamma}) \equiv \Gamma \equiv e^{\alpha(x+ct+\xi_0 - \xi(t) + L_0)}$ will be denoted $e^{\alpha(\cdots)}$ from now on. Then

\begin{alignat*}{1}
\mathcal N\begin{pmatrix} \underline u \\ \underline v \end{pmatrix}_{1} &= 
-c\dot\xi\phi_x + c\dot\xi\tilde\phi_x - \frac{\dot q_u}{\mu} \tilde\phi_x e^{\alpha(\cdots)} - \frac{q_u}{\mu} c\alpha e^{\alpha(\cdots)} + 
\frac{q_u}{\mu} c\alpha\dot\xi e^{\alpha(\cdots)} + \frac{q_u}{\mu} \alpha^2 e^{\alpha(\cdots)} \\ &- q_u e^{\alpha(\cdots)} + q_v e^{\alpha(\cdots)} \\
& \leq -\frac{q_u}{\mu}e^{\alpha(\cdots)}\left(-\omega+c\alpha(1-\dot\xi)-\alpha^2\right) + (q_v-q_u)e^{\alpha(\cdots)}.
\end{alignat*}
Both terms are already negative thanks to the conditions stated in Section 2. Then a computation similar to the preceding section -- thus not detailed here -- leads to
\begin{alignat*}{1}
\mathcal N\begin{pmatrix} \underline u \\ \underline v \end{pmatrix}_{2} \leq 
-q_ve^{\alpha(\cdots)}\left(-\omega+c\alpha(1-\dot\xi)-d/D \alpha^2 - \kappa\right) + f(\tilde\psi).
\end{alignat*}
This quantity can be made negative provided $\omega \leq 2\alpha c$ (which is already the case): indeed, using the exponential decay of $f(\tilde\psi)$ in this zone, the above expression can be factorized as 
$-q_v e^{\alpha(\cdots)} \times \left( \cdot \right)$
with $(\cdot)$ having the sign of $-\omega+c\alpha(1-\dot\xi)-d/D \alpha^2 - \kappa \geq 0$ provided only that $\xi_0$ is large enough (but depending on the initial data only through $\alpha_0$). 

\subsection{$x+ct+\xi_0 \in (-L_0,L_0)$}
\label{smallcc}

First, we ensure $c\xi \leq 1$ by asking that $c\dfrac{B\varepsilon}{\omega}\leq 1$ so by taking
$\varepsilon_0 \leq \di\frac{\omega}{cB}.$ As a consequence, $x+ct+\xi_0-c\xi(t) \in (-L_0-1,L_0)$ and $x-ct-\xi_0+c\xi(t) < -L_0$. Since $\omega < 2\alpha c$, the computations of section \ref{middle} still hold by enlarging the constant $B$ enough.

\subsection{$x+ct+\xi_0 > L_0$}
Here three subcases can appear concerning $x-ct-\xi_0$. By exchanging $\phi,\psi$ and $\tilde\phi, \tilde\psi$ and since $\alpha < \lambda$, the cases $x-ct-\xi_0 \in (-L_0,L_0)$ and $x-ct-\xi_0 > L_0$ are already covered by the computations above. Only the case $x-ct-\xi_0 < -L_0$ remains. In this zone, $x-ct-\xi_0+c\xi(t) < -L_0 + 1$, so here $\min(\Gamma,\tilde\Gamma) \equiv 1$ and both $\psi$ and $\tilde\psi$ are close to~$1$.

Observe that the computations of section 2 still hold by splitting this zone in two subzones: $x<0$ and $x>0$. In the first one, one will bound $f(\psi) + f(\tilde\psi) - f(\underline v)$ by $\text{Lip} f(1-\psi)-\beta(1-\tilde\psi + q_v)$ and in the second one by $\text{Lip} f(1-\tilde\psi) - \beta(1-\psi + q_v)$. Then, since $\omega < \min(\lambda,\tilde\lambda) c$ there holds
$$\mathcal N\begin{pmatrix} \overline u \\ \overline v \end{pmatrix}_{2} \leq -q_v\times(\cdot)$$
with $(\cdot)$ being positive provided $\xi_0$ is large enough. This proves Theorem \ref{propacompact}. 
\end{proof}

\section{Initial data with $O(1)$ compact support}
\label{sectionsmallcc}
We now go back to the original equation \eqref{normal} and state the following.
\begin{thm}
\label{propsmall}
Let $L$ be large enough (independently of $D$). There exist $M', \delta' > 0$ independent of $D > d$  such that if the initial data of \eqref{normal} satisfies $$v_0 > 1-\delta'\text{ for }x\in (-M',M')$$ then after a finite time $t_D = D^{1/2} h(D) + O(1)$ one has $\mu u$ and $v$ satisfying the assumptions of Theorem \ref{propacompact},  i.e. $\mu u, v \geq 1-\delta$ for $x\in(-M \sqrt D,M \sqrt D)$. As a consequence, starting from the time $t=t_D$, propagation occurs as described in Theorem \ref{propacompact}.
\end{thm}
We will divide the proof in several steps:

\noindent{\it Step 1.} Since $D > d$ is ought to be large, $u$ should be very small for small times. Thus we first investigate the equation for $v$ in \eqref{normal} where $u$ is replaced by $0$, and we not only expect to use its solution as a subsolution but we really expect that it will reflect the dynamics of the full solution for some time:
\begin{equation}
\label{robinv}
  \begin{tikzpicture}
  \draw (-6,0) -- (6,0) node[pos=0.5,below] {\small{$d\partial_y \underline v + \underline v = 0$}} node[pos=0.5,above] {};

  \node at (0,-1.25) {$\partial_t \underline v - d\Delta \underline v   = f(\underline v)$};

  \draw (-6,-2.5) -- (6,-2.5) node[pos=0.5,above] {\small{$\partial_y \underline v= 0$}};


  \end{tikzpicture}
\end{equation}

\noindent{\it Step 2.}  Let $p(y)$ be the largest $y$-dependent steady solution to \eqref{robinv}.The travelling wave for \eqref{robinv} connecting $0$ and $p(y)$ will serve to build a subsolution for \eqref{robinv} propagating just as in Theorem \ref{propacompact} but here at speed $c_p = O(1)$. This will give a lower bound on the boundary data $v(x,-L) \geq \underline v(x,-L) \geq \cdots $. This will be the purpose of Lemma \ref{lowerboundbottom}.

\noindent{\it Step 3.} Using this lower bound, we then go back to \eqref{readi}: we show that even without the reaction term, this lower bound suffices to have $\mu u, v \geq 1-\delta$ on $(-M,M)$ within a finite time $t_D$. As a consequence, this is the case also for the nonlinear problem. This will be proved in a final step. Observe that here we use $f\geq 0$. If we were looking for instance at a bistable nonlinearity this would still be true but we would need to add a positive zero-order term in these computations.

We recall the following elementary fact, that we will freely use in the sequel.
\begin{lem}
\label{stat}
There exists $L_0 > 0$ such that if $L > L_0$, there exists a solution $p(y)$ to 
\begin{subequations}
\begin{numcases}{}
-dp'' = f(p)\label{eqp} \\
p'(-L) = 0 \label{pbas} \\
dp'(0) + p(0) = 0 \label{phaut}
\end{numcases}
\end{subequations}
with $p > 0$ concave decreasing and $p(-L) = 1-\delta'' > 1-\delta$. Moreover $\delta'' \to 0$ as $L\to +\infty$.
\end{lem}
Let us now prove the following:
\begin{lem}
\label{lowerboundbottom}
Let $v$ be a solution of \eqref{robinv}. There exist $\delta', M' > 0$ independent of $D$ such that if $v_0 > 1-\delta'$ for $x\in (-M',M')$, there holds $$ v(t,x,-L) \geq (1-\delta'')\varphi_t(x) - C e^{-bt}$$
where $C,b > 0$ do not depend on $D$ and $(\varphi_t)$ is bounded in $\mathcal C^3$ such that $\varphi_t(x) = 1$ for $|x| \leq \frac{c_p}{2} t$ and $\varphi_t(x) = 0$ for $|x| \geq c_p t$ for some speed $c_p > 0$ independent of $D$.
\end{lem}

\begin{proof}
First, that there exists a travelling wave solution with speed $c_p > 0$ independent of $D$ of \eqref{robinv} connecting $0$ and $p(y)$ has to be established: for this we refer to Berestycki-Nirenberg \cite{BN92} which gives the existence of an increasing (in $x$) travelling front $\psi(x,y)$ with exponential convergence towards $0$ and $p(y)$ as $x\to\pm\infty$.

Now we notice that the subsolution argument in Theorem \ref{propacompact} can be used but in a simpler fashion for the Robin homogeneous boundary value problem \eqref{robinv}: one the one hand, the structure of the problem is simpler than the one studied in Theorem \ref{propacompact} since here we deal with a single equation, the original construction of \cite{FML} with $q_v = \varepsilon e^{-\omega t}$ will suffice. On the other hand, $1$ is not a steady state for the problem so one has to replace $1$ by $p(y)$ in the computations. Nonetheless, one can check that the above computations still hold with the adequate subsolution 
$$\psi + \tilde \psi - p - q_v\min(\Gamma,\tilde\Gamma)$$
As a consequence, just as in Theorem \ref{propacompact}, provided $v_0$ is above an initial shift of a pair of waves -- hence the existence of $\delta'$ and $M'$ -- its level lines will be pushed by below by the pair of waves travelling as $\pm c_p t \mp O(1)$. This implies the desired bound.
\end{proof}

\begin{proof}[End of the proof of Theorem \ref{propsmall}]
Let $(u^D,v^D)$ be the solution of \eqref{readi} starting from compactly supported $ 0 \leq \mu u_0, v_0 \leq 1$ and let $v_0$ satisfy the rescaled assumptions of Theorem \ref{propsmall}. First, let $h(D)$ be any positive function such that $h(D)$ grows to infinity as $D\to\infty$, and set $T_D = D^{1/2} h(D)$. We now show the following

\begin{equation}
\label{but}
\underset{t\to +\infty}{\liminf} \underset{D > d}{\inf}\  \underset{(x,y) \in \overline{\Omega_{L,M}}}{\min} \{ \mu u^D(T_D + t,x), v^D(T_D+t,x,y) \} \geq 1 - \delta'' > 1 - \delta
\end{equation}
where $\Omega_{L,M} = (-M,M)\times (-L,0)$. First, it is an easy but tedious exercise to see that the left hand-side of \eqref{but} can be characterised as the limit as $n\to +\infty$ of some $\mu u^{D_n}(T_{D_n} + t_n, x_n)$ or $v^{D_n}(T_{D_n} + t_n, x_n, y_n)$ where $t_n \to +\infty$, $D_n > d$, $(x_n,y_n) \in \overline{\Omega_{L,M}}$. We then extract from $(t_n, D_n, x_n, y_n)$ a subsequence so that $x_n \to x_\infty$ and $y_n \to y_\infty$. Our objective is to extract from $(u,v)$ a subsequence converging to some limiting $(u_\infty, v_\infty)$ to which the maximum principle will apply and force the above limit to be $\geq 1-\delta''$. The difficulty comes from the fact that $(D_n)$ might be unbounded and so that standard parabolic estimates and the usual maximum principle might fall at the limit. Two cases can appear:

\noindent{\it Case 1. $(D_n)$ is unbounded.} Then we extract again so that $D_n \to +\infty$. Let 
$$
u_n(t,x):= u^{D_n}(t_{D_n} + t_n + t, x_\infty + x),\quad
v_n(t,x,y):= v^{D_n}(t_{D_n} + t_n + t, x_\infty + x,y).
$$
Since $f\geq 0$ and by Lemma \ref{lowerboundbottom} above, by the comparison principle we have $(u_n, v_n) \geq (\underline u_n, \underline v_n)$ the solution of 
\begin{equation}
\label{uvn}
  \begin{tikzpicture}
  \draw (-6,0) -- (6,0) node[pos=0.5,below] {\small{$d\partial_y \underline v_n = \mu \underline u_n- \underline v_n$}} node[pos=0.5,above] {$\partial_t \underline u_n - \partial_{xx}^2 \underline u_n = \underline v_n - \mu \underline u_n$};

  \node at (0,-1.25) {$\partial_t \underline v_n - \frac{d}{D_n} \partial_{xx}^2 \underline v_n - d\partial_{yy}^2 v_n  = 0$};

  \draw (-6,-2.5) -- (6,-2.5) node[pos=0.5,above] {\small{$\underline v_n = (1-\delta'')\varphi_{T_{D_n} + t_n}(x_\infty + x) - C e^{-b(t_n + t)} $}};


  \end{tikzpicture}
\end{equation}
Since $d/D_n \to 0$, the standard parabolic estimates applied on $v_n$ will fall concerning the $x$-derivatives. We overcome this difficulty since equation \eqref{uvn} is linear and the boundary data $\underline v_n(t,x,-L)$ is bounded in $\mathcal C^3$: the maximum principle applied on $x$-derivatives of $(u_n,v_n)$ up to order $3$ gives that they are all bounded independently of $n$: $$|\partial^2_{xx} \underline u_n|_\infty, |\partial^3_{xxx} \underline u_n|_\infty, |\partial^2_{xx} \underline v_n|_\infty, |\partial^3_{xxx} \underline v_n|_\infty \leq C_1$$

Now concerning the $y$-derivatives, even though $d/D_n \to 0$ the standard estimates hold: indeed since $\underline v_n\leq 1$, standard $L^p$ parabolic estimates with $p$ large enough applied on $\underline u_n$ give that $\underline u_n$ is bounded in  $\mathcal C^{\alpha,1+\alpha}$ by some $C_2$. Now rescale by $x \leftarrow  x\sqrt{D_n}$ so that $|\underline u_n(t,\frac{x}{\sqrt D_n})|_{\mathcal C^{\alpha, 1+\alpha}} \leq C_3$ (the semi-norms of the derivatives even go to zero since $1/D_n \to 0$). Moreover, under this rescaling $-d/D_n \partial^2_{xx} - d\partial^2_{yy}$ becomes $-d\Delta$ so that standard parabolic estimates up to the Robin boundary apply and give that $|\underline v_n(t,\frac{x}{\sqrt D_n},y)|_{\mathcal C^{1+\alpha/2,2+\alpha}} \leq C_4$. Since this rescaling does not impact $\partial_y$ or $\partial_t$, this gives 
$$|\partial_t \underline v_n|_{\alpha/2}, |\partial_{y} \underline v_n|_\alpha, |\partial^2_{yy} \underline v_n|_\alpha \leq C_4$$

The bound on $\partial^2_{xy} \underline v_n$ follows also by combining the two arguments above, and finally by plugging the estimate on $v$ in the equation for $u$, standard Schauder estimates yield that $\underline u_n$ is bounded in $\mathcal C^{1+\alpha/2,2+\alpha}$. In the end one can extract from $(\underline u_n,\underline v_n)$ some subsequence converging in $\mathcal C^{1,2}_{loc}$ to some $(u_\infty, v_\infty)$ global in time (since $t_n \to +\infty$) solving

\begin{equation}
\label{uvinf}
  \begin{tikzpicture}
  \draw (-6,0) -- (6,0) node[pos=0.5,below] {\small{$d\partial_y v_\infty = \mu  u_\infty - v_\infty$}} node[pos=0.5,above] {$\partial_t u_\infty - \partial_{xx}^2 u_\infty = v_\infty- \mu u_\infty$};

  \node at (0,-1.25) {$\partial_t v_\infty - d\partial_{yy}^2 v_\infty  = 0$};

  \draw (-6,-2.5) -- (6,-2.5) node[pos=0.5,above] {\small{$v_\infty = (1-\delta'') $}};


  \end{tikzpicture}
\end{equation}
Indeed, $v_\infty(t,x,-L) \equiv 1-\delta''$ since $$1 - \delta'' \geq \underline v_n(t,x,-L) \geq 1 - \delta'' - Ce^{-b(t_{D_n} + t_n+t)}$$ for $x \in (-\frac{c_p}{2} h(D_n), \frac{c_p}{2} h(D_n))$ by Lemma \ref{lowerboundbottom} above and by use of $T_{D_n}$.

Since $(u_\infty, v_\infty)$ are global in time, there is no initial data anymore and the maximum principle applies to give 
$$\mu u_\infty, v_\infty \equiv 1-\delta''.
$$
Indeed, no value different than $1-\delta''$ can be reached, because then $(u,v)$ would have an infimum smaller or a supremum larger than $1-\delta''$. By translating over time (which is possible since the solution is global) this infimum or supremum would become a minimum or maximum, that cannot be reached by $u$ because of the strong parabolic maximum principle, and neither by $v$ by the strong parabolic maximum principle and Hopf's lemma applied on the suitable $y$-slice.

\noindent{\it Case 2. $(D_n)$ is bounded.} Then one extracts so that $D_n \to D_\infty > d$ and the above proof is much simpler since standard regularity and maximum principle apply. Moreover $T_D$ is not necessary.

In both cases, the $\liminf$ above is $\geq u_\infty(0,0) = 1 - \delta''$ or $\geq v_\infty(0,0,y_\infty) = 1-\delta'' $, thus $\eqref{but}$ holds. Theorem \ref{propsmall} follows easily: indeed, there exists $t_1$ independent of $D$ such that after $t_D = T_D + t_1$, $\mu u, v > 1-\delta$ on $(-M,M)$.  \end{proof}

\section{Lower bound on the waiting time}\label{lower}
In this subsection, $(u,v)$ denotes the solution of \eqref{readi}. Let $\varepsilon:= D^{-1/2}$ and $v^0$  solve
\begin{equation}
\label{v0}
  \begin{tikzpicture}
  \draw (-6,0) -- (6,0) node[pos=0.5,below] {\small{$d\partial_y v^0 + v^0 = 0$}} node[pos=0.5,above] {};

  \node at (0,-1.25) {$\partial_t v^0 - d\varepsilon^2 \partial^2_{xx} v^0 - d\partial^2_{yy} v^0   = f(v^0)$};

  \draw (-6,-2.5) -- (6,-2.5) node[pos=0.5,above] {\small{$\partial_y v^0= 0$}};

  \end{tikzpicture}
\end{equation}
sharing the same boundary data as $v$: $v^0(0) = v_0$. Observe that $v^0$ is the rescaling of the subsolution $\underline v$ already introduced in equation \eqref{robinv}. The aim of this subsection is to give an estimate on the time during which $v$ is close to $v^0$. More precisely we will show the following

\begin{thm}
\label{borneinftemps}
Let $\alpha\in(0,1)$ and define
$$T_{\alpha, \varepsilon}:= \sup \{\ T > 0 \ \mid \ |v- v^0| < \varepsilon^\alpha \text{ for all } 0<t<T\}.$$  Then for all $0 < \delta < \min\left(\alpha,\frac{2}{7},\frac{2}{5}(1-\alpha)\right)$ one has $$\left(\frac{1}{\varepsilon}\right)^\delta =  \underset{\varepsilon \to 0}{o}(T_{\alpha, \varepsilon})$$
\end{thm}

\begin{rmq}
The limiting case is $\delta = \alpha = 2/7$. This theorem implies Theorem \ref{t2.1}.
\end{rmq}

Let us define $w = v-v^0$. Observe that $(u,w)$ solves 
\begin{equation}
\label{w}
  \begin{tikzpicture}
  \draw (-6,0) -- (6,0) node[pos=0.5,below] {\small{$d\partial_y w + w = \mu u$}} node[pos=0.5,above] {$\partial_t u - \partial^2_{xx} u + \mu u - w = v^0$};

  \node at (0,-1.25) {$\partial_t w - d\varepsilon^2 \partial^2_{xx} w - d\partial^2_{yy} w   = f'(\cdot) w$};

  \draw (-6,-2.5) -- (6,-2.5) node[pos=0.5,above] {\small{$\partial_y w= 0$}};

  \end{tikzpicture}
\end{equation}
where $\cdot \in [v,v^0] \cup [v^0,v]$ (by Taylor's formula). The idea of the proof is to decouple equation \eqref{w} by decomposing $w$ in two parts. Let us set $w=w^1+\overline w$ where 
\begin{equation}
\label{w1}
  \begin{tikzpicture}
  \draw (-6,0) -- (6,0) node[pos=0.5,below] {\small{$d\partial_y w^1 + w^1 = 0$}} node[pos=0.5,above] {};

  \node at (0,-1.25) {$\partial_t w^1 - d\varepsilon^2 \partial^2_{xx} w^1 - d\partial^2_{yy} w^1   = f'(\cdot) w$};

  \draw (-6,-2.5) -- (6,-2.5) node[pos=0.5,above] {\small{$\partial_y w^1= 0$}};

  \end{tikzpicture}
\end{equation}
and 
\begin{equation}
\label{wbar}
  \begin{tikzpicture}
  \draw (-6,0) -- (6,0) node[pos=0.5,below] {\small{$d\partial_y \overline w + \overline  w = \mu u$}} node[pos=0.5,above] {$\partial_t u - \partial^2_{xx} u + \mu u - \overline w = w^1 + v^0$};

  \node at (0,-1.25) {$\partial_t \overline w - d\varepsilon^2 \partial^2_{xx} \overline w - d\partial^2_{yy} \overline w   = 0$};

  \draw (-6,-2.5) -- (6,-2.5) node[pos=0.5,above] {\small{$\partial_y \overline w= 0$}};

  \end{tikzpicture}
\end{equation}
Observe that $T_{\varepsilon,\alpha}$ exists by continuity. We now work by contradiction to show that if $T_{\alpha, \varepsilon} = (1/\varepsilon)^\delta$, then $|w|$ stays of order less than $\varepsilon^{\alpha'}$ with $\alpha' > \alpha$. During the rest of the proof, this will be abbreviated with "$\ll \varepsilon^\alpha$".

The scheme is as follows. First, we derive an  $L^1$ estimate on $w^1$ by Duhamel's formula. This, inserted in estimate \eqref{eA.1} yields the desired estimate on $u$ and then on $\overline w$. We then go back to $w^1$ to obtain the desired estimate, by a more intricate supersolution argument.
\subsection{$L^1$ bound on $w^1$}
By definition, up to time $T_{\varepsilon,\alpha}$ one has $|w| \leq \varepsilon^\alpha$. By Duhamel's formula and the maximum principle for equation \eqref{w1}, this yields
\begin{equation}
\label{firstw1}
|w^1| \leq \varepsilon^\alpha \int_0^t e^{d \varepsilon^2 (t-s)\partial_{xx}}  e^{d(t-s)\partial_{yy}^{NR} }f'(\cdot) \mathrm ds
\end{equation} 
where $\partial_{yy}^{NR}$ denotes $\partial_{yy}$ endowed with the Neuman-Robin boundary condition of \eqref{w1}. Since $\cdot \leq v^0 + \varepsilon^\alpha$, using the above results and rescaling them, one knows that $v^0 \leq \theta$ for $x\leq(a+c_p t)\varepsilon$ for some constant $a > 0$. As a consequence,
 $$ |f'(\cdot)| \leq \text{Lip} f \times \mathbf 1_{((-a-c_ps)\varepsilon,(a+c_ps)\varepsilon)}.
 $$
Also, by the maximum principle, there exist $C(d) > 0$ and $\lambda_1(d) > 0$ such that 
$$e^{d(t-s)\partial_{yy}^{NR} } \mathbf 1_{((-a-c_ps)\varepsilon,(a+c_ps)\varepsilon)} \leq Ce^{-\lambda_1(t-s)}\mathbf 1_{((-a-c_ps)\varepsilon,(a+c_ps)\varepsilon)}.
$$
Using both estimates in \eqref{firstw1}, the maximum principle yields
\begin{equation}
\label{secondw1}
|w^1| \leq C\text{Lip} f  \varepsilon^\alpha \int_0^t e^{d \varepsilon^2 (t-s)\partial_{xx}}  e^{-\lambda_1(t-s)}\mathbf 1_{((-a-c_ps)\varepsilon,(a+c_ps)\varepsilon)} \mathrm ds.
\end{equation} 
Since $e^{d \varepsilon^2 (t-s)\partial_{xx}} $ preserves the $L^1$ norm, this gives the estimate, for some constants $C_1,C_2$ that do not depend on $\varepsilon$:
\begin{equation}
\label{w1L1}
|w^1(\cdot,y)|_{L^1(\mathbb R)} \leq C\text{Lip} f  \varepsilon^\alpha  \int_0^t e^{-\lambda_1(t-s)} 2(c_ps + a)\varepsilon \mathrm ds = \varepsilon^{\alpha+1}(C_1 + C_2t) \ll \varepsilon^\alpha,
\end{equation}
since $\delta < 1$.	

\subsection{Estimate on $u$ and $\overline w$}
Using the appendix estimate \eqref{eA.1} and Duhamel's formula one gets
\begin{alignat}{1}
\label{firstterm}
|u| \leq\  &C_3 \int_0^t \left(\int_\mathbb{R}  \frac{e^{-(x'-x)^2/(4a(t-s))}}{\sqrt{4\pi a (t-s)}} |w^1+v^0|(s,x',0) \mathrm dx'\right) \mathrm ds \\ 
\label{secondterm} + &C_4\int_0^t  e^{-\omega (t-s)} |(w^1+v^0)(s,\cdot,0)|_{L^1(\mathbb R)} \mathrm ds.
\end{alignat}
First observe that due to the above results and the rescaling, one has 
$$|v^0(s,\cdot,y)|_{L^1} \leq (C_5+c_p s)\varepsilon.
$$ Using this and estimate \eqref{w1L1}, we deal with the second term: 
$$\eqref{secondterm}  \leq C_4\int_0^t e^{-\omega (t-s)} \left(  \varepsilon^{\alpha+1}(C_1+C_2s) + \varepsilon(C_5+c_ps) \right) \mathrm ds \leq (C_6 + C_7t)(\varepsilon^{\alpha+1} + \varepsilon) \ll  \varepsilon^\alpha$$
since $1-\delta > \alpha$.
We now deal with the first term:
$$\eqref{firstterm} \leq C_3\int_0^t \frac{1}{\sqrt{4\pi a (t-s)}} |(w^1+v^0)(s,\cdot,0)|_{L^1(\mathbb R)} \mathrm ds \leq (C_8t^{1/2} + C_9t^{3/2})(\varepsilon^{\alpha+1} + \varepsilon)\ll \varepsilon^\alpha$$
since $1-\frac{3}{2}\delta > \alpha$.
As a consequence, $|u| \ll  \varepsilon^\alpha$.
Now seeing equation \eqref{wbar} as a boundary value problem for $\overline w$, we see that the above estimate on $\mu u$ provides an easy supersolution that stays above $\overline w$, that is
$$\overline w \leq \mu\left( C_6 + C_7t + C_8t^{1/2} + C_9t^{3/2} \right)\left( \varepsilon^{\alpha+1} + \varepsilon \right) \ll \varepsilon^\alpha.$$

\subsection{Back to $w^1$}
Using $w = w^1 + \overline w$ we rewrite equation \eqref{w1} as a linear non-homogeneous problem:
\begin{equation}
\label{w1lin}
  \begin{tikzpicture}
  \draw (-6,0) -- (6,0) node[pos=0.5,below] {\small{$d\partial_y w^1 + w^1 = 0$}} node[pos=0.5,above] {};

  \node at (0,-1.25) {$\partial_t w^1 - d\varepsilon^2 \partial^2_{xx} w^1 - d\partial^2_{yy} w^1 - f'(\cdot)w^1  = f'(\cdot) \overline w$};

  \draw (-6,-2.5) -- (6,-2.5) node[pos=0.5,above] {\small{$\partial_y w^1= 0$}};

  \end{tikzpicture}
\end{equation}	
Since $w^1(0) = 0$, Duhamel's formula gives
\begin{equation}\label{duhamelsursol} w^1(t) = \int_0^t w_h^s(t-s) \mathrm ds \end{equation}
where $w_h^s$ solves
\begin{equation}
\label{wh}
  \begin{tikzpicture}
  \draw (-6,0) -- (6,0) node[pos=0.5,below] {\small{$d\partial_y w_h^s + w_h^s = 0$}} node[pos=0.5,above] {};

  \node at (0,-1.25) {$\mathcal L w_h^s:= \partial_t w_h^s - d\varepsilon^2 \partial^2_{xx} w_h^s - d\partial^2_{yy} w_h^s - f'(\cdot)w_h^s  = 0$};

  \draw (-6,-2.5) -- (6,-2.5) node[pos=0.5,above] {\small{$\partial_y w_h^s = 0$}};

  \end{tikzpicture}
\end{equation}	
along with the initial condition $w_h^s(0) = (f'(\cdot) \overline w)(s).$
Inspired by the linearisation and a rescaling of the supersolution to the non-linear equation in the previous section, and using the same notations as in it, we look for a supersolution with the form
$$\overline{w_h^s}(t) = \xi(t) \partial_x\psi(x+c_p\varepsilon t + x_0^s) - \xi(t)\partial_x\tilde{\psi}(x-c_p\varepsilon t - x_0^s) + Ce^{-\omega t}\min(\Gamma, \tilde{\Gamma})$$
for some $\xi(t)$ increasing in time and initial shift $x_0^s$.
First of all, we need to ensure ordering of the initial data. That is:
$ \overline{w_h^s}(0) \geq (f'(\cdot) \overline w)(s)$,
which is obtained provided $$\overline{w_h^s}(0) \geq \text{Lip} f \varepsilon^{\alpha+\nu} \mathbf 1_{((-a-c_ps)\varepsilon,(a+c_ps)\varepsilon)}$$ for some 
$\nu < 1 -{3}/{2}\delta - \alpha$ (since $\bar w \leq K \varepsilon^{1-\frac{3}{2}\delta}$ thanks to the computations above). We achieve this by asking
\begin{equation}
 \label{condx0s}
C = (\text{Lip} f) \varepsilon^{\alpha+\nu},\quad x_0^s > (a+c_p)s\varepsilon + L_0.
\end{equation}
We will also see below that we need  $\delta < \nu$.
Combining both these conditions imposes $\delta < 1-\frac{3}{2}\delta-\alpha$, i.e. the assumption $\delta < \frac{2}{5}(1-\alpha)$ of Theorem \ref{borneinftemps}.

Now, straightforward computations give
$$
\begin{array}{rll}
\mathcal L  \overline{w_h^s} =  &\xi f'(\psi) \partial_x\psi - \xi f'(\tilde\psi) \partial_x \tilde\psi  + \dot\xi \partial_x\psi - \dot\xi\partial_x \tilde\psi  - f'(\cdot)\left(\xi \partial_x\psi - \xi \partial_x\tilde\psi \right) - \omega Ce^{-\omega t}\min(\Gamma,\tilde\Gamma)\\ 
& + Ce^{-\omega t} \partial_t \min(\Gamma,\tilde\Gamma) + C e^{-\omega t} \partial_{xx} \min(\Gamma,\tilde\Gamma) - f'(\cdot) Ce^{-\omega t} \min(\Gamma,\tilde\Gamma).
\end{array}
$$
As in Section \ref{wavelikes}, we analyse the sign of this quantity in three separate zones.
Observe that due to the rescaling between \eqref{normal} and \eqref{readi}, decay exponents $\Theta,\Theta_0$ (resp. the $\alpha$ and $\alpha_0$ from Sec. \ref{wavelikes}) and $\lambda,\tilde\lambda$ scale here as $1/\varepsilon$ as well as the lower bounds on the derivatives $\delta_L$. Remember also that we are looking only at times $t \leq \varepsilon^{-\delta},$ so we only need to find a supersolution up to this time. We also reinitialize the constants $C_i$ and $K_i$ which will be positive constants independent of $\varepsilon$.

\subsubsection{$x+c_p\varepsilon t + x_0^s < -L_0$}
As before, in this zone we have $\psi \leq \theta/2$, $\tilde\psi \geq (1+\theta_1)/2$, $\psi, \partial_x\psi \leq C_1 e^{-\lambda(x+c_p\varepsilon t + x_0^s + L_0)}$, $1-\tilde\psi, \partial_x\tilde\psi \leq C_1 e^{-\tilde\lambda(x-c_p\varepsilon t - x_0^s + L_0)}$ and $\min(\Gamma,\tilde\Gamma) \equiv \Gamma = e^{\Theta(x+c_p\varepsilon t + x_0^s)} =: e^{\Theta(\cdots)}$. As a consequence and since $\Theta < \tilde\lambda$, one has

$$\mathcal L  \overline{w_h^s} \geq \dot\xi\partial_x\psi - \dot\xi \partial_x\tilde\psi + \text{Lip} f \varepsilon^{\alpha+\nu}e^{-\omega t}e^{\Theta(\cdots)}\left[\left(-\omega + c_p \Theta\varepsilon + \Theta^2 \right) - \frac{\xi}{\varepsilon^{\alpha+\nu}}e^{(\omega-2\Theta c_p\varepsilon)t} e^{-2x_0^s \Theta}  \right].
$$

The first term inside the brackets is positive provided $\omega < \Theta^2$ (which is not a constraint since $\Theta$ grows as $1/\varepsilon$) and we can make the whole bracket positive provided \begin{equation}
\label{contraintew}
\omega < 2 \Theta c_p \varepsilon
\end{equation}
(observe that the right-hand side in \eqref{contraintew} is bounded from above and by below by positive constants that do not depend on $\varepsilon$) by taking
\begin{equation}
\label{x0sxi}
x_0^s > -\frac{1}{2\Theta} \ln \left(\left(-\omega + c_p \Theta\varepsilon + \Theta^2 \right) \frac{\varepsilon^{\alpha+\nu}}{\xi(\varepsilon^{-\delta})} \right).
\end{equation}
This will be a constraint on our future choice of $\xi(t)$, to be kept in mind.
\subsubsection{$x+c_p\varepsilon t + x_0^s \in (-L_0,L_0)$}
Here $x-c_p \varepsilon t - x_0^s < -L_0$, and $\tilde\psi \geq (1+\theta_1)/2$, $1-\tilde\psi, \partial_x\tilde\psi \leq C_2 e^{-\tilde\lambda(x-c_p\varepsilon t - x_0^s + L_0)}$, $f'(\tilde\psi) \leq -\beta$ and $\min(\Gamma,\tilde\Gamma) = \Gamma$.

\begin{alignat*}{1}
\mathcal L  \overline{w_h^s} \geq\  &\xi f'(\psi) \partial_x \psi - \xi f'(\tilde\psi) \partial_x\tilde\psi + 2\delta_{2L_0} \dot\xi - \varepsilon^{\alpha+\nu} e^{-\omega t} \left( \omega + |\Gamma|_{\mathcal C^2} \right) \\
-&f'(\cdot)\left[ \xi\partial_x\psi - \xi\partial_x\tilde\psi \right] - f'(\cdot) \varepsilon^{\alpha + \nu} e^{-\omega t} \Gamma \\
\geq\  & \xi f'(\psi)\partial_x\psi - \xi f'(\tilde\psi) \partial_x\psi + 2\delta_{2L_0} \dot\xi - \varepsilon^{\alpha+\nu} e^{-\omega t} \underbrace{\left( \omega + |\Gamma|_{\mathcal C^2} + \text{Lip}(f')|\Gamma|_{\mathcal C^2} \right)}_{C_3}\\ - &f'(\cdot)\left[ \xi\partial_x\psi - \xi\partial_x\tilde\psi \right]
\end{alignat*}
We make this positive by counterbalancing the negative terms thanks to  $\delta_{2L_0}\dot\xi$, by asking
\begin{equation}
\label{xik1}
\dot\xi \geq \frac{2}{\delta_{2L_0}} \varepsilon^{\alpha+\nu} e^{-\omega t} C_3 =: K_1 \varepsilon^{\alpha+\nu+1} e^{-\omega t} 
\end{equation}
Moreover by using $\Theta < \tilde\lambda$, $\omega < 2\Theta c_p \varepsilon$ and the previous expression of $e^{-2\Theta x_0^s}$ one obtains
$$-\xi f'(\cdot) \partial_x\tilde\psi, -\xi f'(\tilde\psi) \partial_x\tilde\psi \geq -\varepsilon^{\alpha+\nu} C_{4} e^{-\omega t}$$ so that we ask also for 
\begin{equation*}
\dot\xi \geq K_2 \varepsilon^{\alpha+\nu +1} e^{-\omega t},
\end{equation*}
which is implied by \eqref{xik1} by taking $K_1$ large enough.
The last term to counterbalance is 
\begin{alignat*}{1}
\xi\partial_x\psi [f'(\psi) - f'(\cdot)]  \geq -\xi\partial_x\psi \text{Lip}(f)\left( |v_0 - \psi| + \varepsilon^\alpha  \right)
\end{alignat*}
by the triangle inequality and the definition of $\cdot$. But we also know that 
$$
|\psi - v_0| \leq |\psi - (\psi + \tilde\psi - 1)| + |(\psi + \tilde\psi - 1) - v_0|  
\leq  |1-\tilde\psi| + C_{5} e^{-\omega_0 t}
$$
by Section \ref{sectionsmallcc} for some $C_{5}, \omega_0 > 0$ independent of $\varepsilon$. Now just as above, one can use the exponential decay of $1-\tilde\psi$ in the current zone to prove that there exists $C_{6} > 0$ such that $|1-\tilde\psi| \leq C_{6} e^{-\omega t}$. In the end 
$$ \xi\partial_x\psi [f'(\psi) - f'(\cdot)]  \geq - \text{Lip}(f) \xi \partial_x\psi \left[ C_{6} e^{-\omega t} + C_{5} e^{-\omega_0 t} + \varepsilon^\alpha \right].
$$
We can reduce $\omega$ and change the constants so that 
$$ \xi\partial_x\psi [f'(\psi) - f'(\cdot)]  \geq - \text{Lip}(f) \xi \partial_x\psi \left[ C_{7} e^{-\omega t} + \varepsilon^\alpha \right]$$
and we counterbalance this term by asking (remember the additional power of $\varepsilon$ factor due to the scaling of $\delta_{2L_0}$) 
\begin{equation}
\label{xik2}\dot\xi \geq K_2 \varepsilon \xi e^{-\omega t},\quad 
\dot\xi \geq K_3 \xi \varepsilon^{\alpha+1}  
\end{equation}
We now have to find a suitable increasing function $t\mapsto \xi(t)$ satisfying \eqref{xik1}, \eqref{xik2} and that should not increase too much so that $w^1(\varepsilon^{-\delta}) \ll \varepsilon^{\alpha}$. Since the order between the right-hand sides in \eqref{xik2}  changes at some point in time, we define $\xi$ in two parts as a continuous but only piecewise $\mathcal C^1$ function. This is not a problem since one can apply the maximum principle a second time starting from the junction. We propose 

$$\xi(t):= \left\{
	\begin{array}{ll}
		\varepsilon^{\alpha + \nu} e^{-K_2 \frac{\varepsilon}{\omega} e^{-\omega t}} & \mbox{if } t < \frac{1}{\omega} \ln\left( \frac{K_3}{K_2} \varepsilon^{-\alpha} \right) \\
		B(\varepsilon) \varepsilon^{\alpha+\nu} e^{K_3 \varepsilon^{\alpha+1}t} & \mbox{if } t \geq \frac{1}{\omega} \ln\left( \frac{K_3}{K_2} \varepsilon^{-\alpha} \right)
	\end{array}
\right.$$
with $B(\varepsilon) > 0$ uniformly bounded from above and by below in $\varepsilon$ is chosen so that $\xi$ is continuous: $B(\varepsilon) = e^{K_3 \varepsilon^{\alpha+1} \frac{1}{\omega} \ln\left( \frac{K_3}{K_2} \varepsilon^{-\alpha}\right)} e^{K_2 \frac{\varepsilon}{\omega}\frac{K_2}{K_3} \varepsilon^\alpha}$. 
Observe that \eqref{xik2} is automatically satisfied since we just integrated the stronger differential equation between the two on the associated time-intervals. Observe that \eqref{xik1} is indeed satisfied provided $K_2, K_3 > K_1$ and $\varepsilon$ is small.

Now with this choice of $\xi$ -- since $\xi \ll \varepsilon^{\alpha+\nu}$ up to time $\varepsilon^{-\delta}$ -- observe that the remaining condition on $x_0^s$ \eqref{x0sxi} is void since $x_0^s > 0$ and it reduces to the initial one \eqref{condx0s}. Finally, the only condition on $\omega$ is \eqref{contraintew}.
\subsubsection{$x+c_p\varepsilon t + x_0^s  \geq L_0$}
As before, we deal with this last zone by using symmetry and by repeating the arguments above: no stronger condition appears and the computations above hold (by eventually changing the constants).

\subsection{End of the proof of Theorem \ref{borneinftemps}}
We now estimate $w_1$ thanks to the supersolution. Coming back to \eqref{duhamelsursol} with $t = \varepsilon^{-\delta}$ one gets
$$w_1(t) \leq  K_6 \int_0^t \xi(s) \mathrm ds + K_7 \int_0^t \varepsilon^{\alpha+\nu} e^{-\omega s} \mathrm ds$$
The second term is bounded by $K_7 \varepsilon^{\alpha+\nu}/\omega \ll \varepsilon^\alpha$. For the first term, we divide the integral in two parts. Call $t_{j} = \frac{1}{\omega} \ln\left( \frac{K_3}{K_2} \varepsilon^{-\alpha}\right)$ the junction time.
$$
\int_0^t \xi(s) \mathrm ds \leq \int_0^{t_j} \xi(s) \mathrm ds + \int_{t_j}^t \xi(s) \mathrm ds  \leq K_8 \varepsilon^{\alpha + \nu}\ln(\varepsilon^{-\alpha}) + \frac{B \varepsilon^{\alpha+\nu}}{K_3 \varepsilon^{\alpha+1}} \left[ e^{K_3 \varepsilon^{\alpha+1} t} - 1 \right]
$$
by using a crude upper bound for the first part in the definition of $\xi(t)$.
Now since $\delta < \alpha + 1$, $$\frac{B \varepsilon^{\alpha+\nu}}{K_3 \varepsilon^{\alpha+1}} \left[ e^{K_3 \varepsilon^{\alpha+1} t} - 1 \right] \underset{\varepsilon\to 0}{\sim} B\varepsilon^{\alpha+\nu} t = B\varepsilon^{\alpha+\nu - \delta} \ll \varepsilon^\alpha$$ because $\nu > \delta$. 
In the end, both $\bar w$ and $w_1$ are $\ll \varepsilon^\alpha$ up to time $t = \varepsilon^{-\delta}$ so we have a contradiction.

\section{Initial data supported on the road only}
In this section we investigate the behaviour of solutions starting from $(u_0,v_0) = (\mathbf 1_{(-a,a)},0)$. We still denote  $\varepsilon:= 1/\sqrt{D}$.
\subsection{$a$ is small}
\begin{thm}
\label{u0petit}
There exists $a_0 > 0$ such that for $a < a_0$, the solution of \eqref{readi} starting from $(\mathbf 1_{(-a,a)},0)$ decays to $0$ uniformly.
\end{thm}

The proof relies on a suitable reformulation of equation \eqref{normal} and a crude  bound on $f$.
Observe that, if we
replace $v$ by its even extension on $\mathbb R\times [-L,L]$, we have 
\begin{equation}
\begin{cases}
\partial_t v - d\varepsilon^2 \partial_{xx}^2 v - d\partial_{yy}^2 v = f(v) + 2d(\mu u-v)(x,0)\mathrm d\lambda_{y=0} \\
\partial_t u - \partial_{xx}^2 u +  \mu u = v(x,0)
\end{cases}
\end{equation}
where $\mathrm d\lambda_{y=0}$ denotes the Lebesgue measure on the line $\{y=0\}$.

\begin{lem}
\label{vsqrtt}
Let $C = \max(\text{Lip} f, 2d)$. Then 
$$v(t,x,y) \leq C(t + 2C'\sqrt{t})$$
where $C'$ is a constant that depends only on $d$ and $L$.
\end{lem}

\begin{proof}

%
%
This is basically an Aronson type inequality (see \cite{AronsonEst}), we give a quick computation here. By Duhamel's formula,
\begin{equation}
\label{duhamel}
v(t,x,y) = \int_0^t e^{s\Delta{d,\varepsilon}^N} [f(v(t-s,x,y)) + 2d(\mu u -v)(t-s,x,0)\mathrm d \lambda_{y=0}] \mathrm ds
\end{equation}
where $\Delta_{d,\varepsilon}^N = d \varepsilon^2 \partial_{xx}^2 + d\partial_{yy}^2$ endowed with Neumann boundary conditions on $y=\pm L$. Since $\mathbb R \times (-L,L)$ is a product domain and since $\partial_{xx}^2$ and $\partial_{yy}^2$ commute, we can compute this heat kernel as follows. 
Denote $\lambda_k = d(k\pi /(2L))^2$ the eigenvalues of $d\partial_{yy}^2$ on $(-L,L)$ with Neumann conditions and $\phi_k$ the associated eigenfunctions. Then the  heat kernel is 
$K_2(t,y,y') =\displaystyle \sum_{k\geq 0} e^{-\lambda_k t} \phi_k(y) \phi_k(y'). 
$ 
For $d\varepsilon^2 \partial_{xx}^2$ on $\mathbb R$ the heat kernel is  
$\displaystyle K_1(t,x,x') = \frac{1}{\sqrt{4\pi d\varepsilon^2 t}} e^{-(x-x')^2/(4d\varepsilon^2 t)}.
$
As a consequence, 
$$
e^{s\Delta{d,\varepsilon}^N} \mathrm d \lambda_{y=0} = \int_{\mathbb R} K_1(s,x,x') K_2(s,y,0) \mathrm dx' = \sum_{k\geq 0} e^{-\lambda_k s} \phi_k(y)\phi_k(0),
$$
which of course depends only on $y$ and is even in $y$ (the $\phi_k$ being even or odd). Observe that this is nothing more than the fundamental solution of the diffusion equation in $y$ on $(-L,L)$. Because the $\phi_k$ are uniformly bounded by  $C'$ depending only on $d$ and $L$ one gets

$$ e^{s\Delta{d,\varepsilon}^N} \mathrm d \lambda_{y=0} \leq C' \sum_{k\geq 0} e^{-\lambda_k s} \leq C'/\sqrt{s}$$ for another constant $C'$. The last inequality comes from the growth of $\lambda_k$ as $Ck^2$.

Going back to \eqref{duhamel} and using $f(v) \leq \text{Lip} f$ as well as $\mu u - v \leq 1$ and the positivity of the integral, one gets 

$$ v(t,x,y) \leq C\int_0^t (1+C'/\sqrt{s}) \mathrm ds \leq C(t+2C' \sqrt t),
$$
which implies the lemma.
\end{proof}

\begin{lem}
We have
$$u(t,x) \leq \frac{2 e^{-\mu t}}{\sqrt{4\pi t}}a + \frac{Ct^2}{2} + \frac{4CC'}{3}t^{3/2}.
$$
\end{lem}

\begin{proof}
We insert the previous estimate on $v(t,x,0)$ in the equation satisfied by $u$ and solve it using Duhamel's formula. By the maximum principle, this gives the following upper bound:
\begin{alignat*}{1}
u(t,x) & \leq e^{-\mu t}e^{t\partial_{xx}^2}u_0 + C\int_0^t e^{-\mu (t-s)}e^{(t-s)\partial_{xx}^2}(s+2C'\sqrt s) \mathrm ds \\
& = e^{-\mu t} \int_{-a}^a \frac{1}{\sqrt{4\pi t}}e^{-\frac{(x-x')^2}{4t}} \mathrm dx' + C\left(\frac{t^2}{2} + \frac{4}{3}C't^{3/2}\right)
\end{alignat*}
which gives the desired result.
\end{proof}

\begin{proof}[Proof of Theorem \ref{u0petit}]
Chose $t_1'$ such that $C(t_1'^2 + 2C' t_1'^{3/2}) =  \frac{\theta}{2}$ and set $t_1 =  \max(1,t_1')$. As a consequence, at time $t=t_1$ one has $v \leq \frac{\theta}{2}$ and 
$$ \mu u(t,x) \leq 2e^{-\mu t_1}\frac{1}{\sqrt{4\pi t_1}}a + \frac{\theta}{2} \leq \frac{2\theta}{3},
$$
 if $a < a_0$ for some $a_0$. Then the maximum principle yields that from this time $\mu u, v$ will always stay below the constant solution $2\theta/ 3$ of \eqref{normal}. And so,  $\mu u$ and $v$ will tend to $0$.
\end{proof}

\subsection{Best case scenario: $a=+\infty$}
In this subsection we take $\mu u_0 \equiv 1$. Since both the initial data and equation \eqref{readi} enjoy here a translation invariance in the $x$ direction, $u$ and $v$ do not depend on $x$. We prove the following

\begin{thm}
\label{theoremmu}
There exist $\mu^\pm > 0$ such that:
\begin{enumerate}[a)]
\item If $\mu > \mu^+$, $\mu u$ and $v$ converge uniformly to $1/(\mu(L+1/\mu))$ as $t\to +\infty$.
\item If $\mu < \mu^-$, $\mu u$ and $v$ converge uniformly to $1$ as $t\to +\infty$.
\end{enumerate}
\end{thm}

\begin{proof}[Proof of point a)]
Using Lemma \ref{vsqrtt} one gets,  for $t\leq 1$: $v(t,x,y) \leq C \sqrt t$ (for some constant $C$ different than the $C$ in the afore mentioned Lemma).
Using this in the equation for $u$, one gets 
$$\mu u(t,x) \leq e^{-\mu t} + C\mu \int_0^t e^{-\mu(t-s)} \sqrt s\  \mathrm ds \leq e^{-\mu t} + C\mu t^{3/2}. 
$$
So, at 
$t_\mu = \displaystyle\left(\frac{\theta/(2C)}{\mu}\right)^{2/3},
$
provided $\mu$ is large enough so that $e^{-\mu t_\mu} \leq \theta/2$ and $v \leq \theta$, i.e.
$$ \mu \geq \max\left( \left(\frac{2C}{\theta}\right)^2 | \ln(\theta /2) |^3, \frac{1}{2} \left( \frac{C}{\theta}  \right)^2 \right) =: \mu^+,
$$
one has $\mu u, v \leq \theta$. By the comparison principle, this will hold for all $t > t_\mu$ and $v$ never gets above $\theta$ anywhere.
As a consequence, $\mu u(t), v(t,y)$ converge to a common limit $l \leq \theta$ satisfying $(L+1/\mu)l = 1/\mu$ (conservation of mass).
\end{proof}
\begin{proof}[Proof of point b)]
The idea of the proof is simple: we investigate whether the sole diffusion is able to transfer enough mass from $u$ to $v$ so that in finite time $v$ is above $\theta$ on a large enough interval $(-L_0,0)$. The quantity $L_0$  is linked to Kanel' and Aronson-Weinberger  \cite{KanelFinite,AW}. Using $v \geq 0$ and the strong parabolic maximum principle one gets
$ \mu u \geq e^{-\mu t}$.
So that setting $\theta' = (1+\theta)/2$ and $t_M = \frac{1}{\mu}\ln\left(\frac{1}{\theta'}\right)$ one has $\mu u \geq \theta'$ while $t \leq t_M$ so that, by the maximum principle, Hopf's lemma and the positivity of $f$, up to time $t_M$ we have $v \geq \underline v$ the solution of 

\begin{equation}
\label{soussoltM}
  \begin{tikzpicture}
  \draw (-6,0) -- (6,0) node[pos=0.5,below] {\small{$d\partial_y \underline v + \underline v = \theta' $}} ;

  \node at (0,-1.25) {$ \partial_t \underline v  - d\partial_{yy}^2 \underline v - d\varepsilon^2 \partial_{xx}^2\underline v   = 0$};

  \draw (-6,-2.5) -- (6,-2.5) node[pos=0.5,above] {\small{$\partial_y\underline v  = 0$}};

  \end{tikzpicture}
\end{equation}
starting from $\underline v_0 = 0$. Observe that  $\underline v$ is independent of $x$, so  we will call it $\underline v(t,y)$ from now on. The function $w = \theta' - \underline v$ is easily decomposed as
$$w(t,y) = \sum_{k \geq 0} e^{-\lambda_k t} \tilde w_k(0) \cos\left(\sqrt{ \frac{\lambda_k}{d}}(y+L)\right),
$$
where the $\frac{\lambda_k}{d} > 0$ solve $\sqrt x = \text{cotan}(\sqrt x L)$ and $\displaystyle\sum \tilde w_k(0) \cos\left(\sqrt{ \frac{\lambda_k}{d}}(y+L)\right) = \theta'.$
By the maximum principle we have, for $y\in (-L_0,0)$:
$$w \leq  \theta' e^{-\lambda_1 t} \frac{\cos\left(\sqrt{\lambda_1/d}\ (L-L_0)\right)}{\cos\left(\sqrt{\lambda_1/d}\ L\right)} =: K e^{-\lambda_1 t},$$
so that for $y\in (-L_0,0), \  v(t_M,y) \geq \theta' - K e^{-\lambda_1 t_M} \geq \frac{1+3\theta}{4}$ provided 
\begin{equation}
\label{estimu}
\mu \leq \lambda_1 \frac{\ln(1/\theta')}{\ln\left( \frac{4K}{1-\theta} \right)} =:\mu^-
\end{equation}
%
Chose $L_0$ large enough in the beginning so that an initial condition $\mu u(t_M), v(t_M,y) \geq (1+3\theta)/4$, for all $y\in (-L_0,0)$, leads to invasion: $\mu u, v \to 1$ as $t \to \infty$. The existence of such an $L_0$ follows from Kanel', Aronson-Weinberger \cite{KanelFinite, AW} on $\mathbb R$. In our context, it is in fact simpler since total mass is confined in $(-L,0)$ and a single point whereas in \cite{KanelFinite, AW} it can be spread on all $\mathbb R$.
\end{proof}

\subsection{Large $a < +\infty$}
We use the best case scenario described above to prove the existence of large but finite $a$ that will lead to invasion.
Our proof relies on the fact that $\mathbf 1_{(-a,a)}$ and $\mathbf 1_{(-\infty,\infty)}$ are close in $L^\infty$ weighted by some $\rho(x)$ with tails $e^{-|x|}$ and that such a weight preserves the semi-linear parabolic and monotone structure of the system \eqref{readi}. In particular, the "weighted equation" will have a locally (in time) Lipschitz continuous flow. Going back to the original solutions, this Lipschitz continuity becomes a uniform continuity on every compact subset.  

\begin{lem}
\label{lemweight}
There exists a smooth weight $\rho(x)>0$ such that $\rho(x) = e^{-|x|}$ for $|x| > 1$ and such that the following holds.
Denote $\| (f,g) \|_X =  \max\left( \| \rho f \|_{L^\infty(\mathbb R)}, \| \rho g \|_{L^\infty(\Omega_L)}\right).$
For every $T,M> 0$, there is  $C_{T,M} > 0$ independent of  $D$  such that
$$ \sup_{0 \leq t \leq T, x\in(-M,M)} \left( |u- \tilde u| + |v - \tilde v| \right) \leq C_{T,M} \| (u_0 - \tilde u_0, v_0 - \tilde v_0)\|_X $$
for every $(u,v)$ and $(\tilde u, \tilde v)$ solutions of \eqref{readi} starting from respectively $(u_0,v_0)$ and $(\tilde u_0, \tilde v_0)$. 
\end{lem}
\begin{rmq}
Observe  that the above Lemma could be stated for any $\rho_\alpha(x) = e^{-\alpha |x|}$ (with $\alpha > 0$) by changing the constants: this is due to the scaling invariance $(t,x,y) \to (\Lambda t, \sqrt \Lambda x, \sqrt \Lambda y)$ of equation \eqref{readi} ; indeed, $\rho_\alpha$ becomes $\rho_1$ in the rescaling by $\Lambda = \alpha^2$.
\end{rmq}
\vspace{5pt}
\begin{proof}[End of the proof of Theorem \ref{thmintroproparoutel}]
Once Lemma \ref{lemweight} is proved, the end of Theorem \ref{thmintroproparoutel} follows easily. Indeed, let $u_0 = \mathbf 1_{(-a,a)}$ and $v_0 = 0$ as well as $\tilde u_0 = \mathbf 1_{(-\infty,+\infty)}$ and $\tilde v_0 = 0$. Observe that if $a > 1$, we have
$\| (u_0 - \tilde u_0, v_0 - \tilde v_0)\|_X  = e^{-a}.$
Moreover, since $\mu < \mu^-$, by Theorem \ref{theoremmu} above, there exists $T > 0$ such that 
$\mu \tilde u(T,x), \tilde v(T,x,y) \geq 1-\delta/2,
$ 
with $\delta$ as in Theorem \ref{propacompact}.
By choosing $a > \max \{1, -\ln(\delta/(2C_{T,M}))\} =: a_1$ (which does not depend on $D$), and applying Lemma \ref{lemweight} on $[0,T]\times[-M,M]$ (with $M$ as in Theorem \ref{propacompact}), one has 
$$ |\mu u(T,x) - \mu \tilde u(T,x)| + |v(T,x,y) - \tilde v(T,x,y)| \leq \delta/2 \text{ for all } -M<x<M.$$ 
As a consequence, 
$ \mu u(T,x), v(T,x,y) \geq 1-\delta$ for all $-M<x<M$.
\end{proof}
\vspace{5pt}
\begin{proof}[Proof of Lemma \ref{lemweight}]

The proof relies only on the parabolic maximum principle applied to a weighted equation.
Let $\rho(x)$ define a positive $\mathcal C^2$ function such that $\rho(x) = e^{-|x|}$ for $|x|\geq 1$. Let $(u,v)$ solve system \eqref{readi}. Observe that 
$(\mathfrak u, \mathfrak v):= (\rho u, \rho v)$
satisfies
\begin{equation}
\label{weightedreadi}
  \begin{tikzpicture}
  \draw (-6,0) -- (6,0) node[pos=0.5,below] {\small{$d\partial_y \mathfrak v = \mu \mathfrak u - \mathfrak v $}} node[pos=0.5,above] {$\partial_t \mathfrak u + 2\frac{\rho'}{\rho}\partial_x \mathfrak u - \partial_{xx}^2 u  = \mathfrak v - \left(\mu + \frac{\rho''}{\rho} - 2\left( \frac{\rho'}{\rho} \right)^2 \right) \mathfrak u$};

  \node at (0,-1.25) {$\partial_t \mathfrak v + 2\frac{d\rho'}{D\rho}\partial_x \mathfrak v - \frac{d}{D} \partial_{xx}^2 \mathfrak v - d\partial_{yy}^2 \mathfrak v  = \rho f\left(\frac{\mathfrak v}{\rho}\right) - \frac{d}{D} \left( \frac{\rho''}{\rho} - 2\left( \frac{\rho'}{\rho} \right)^2 \right) \mathfrak v$};

  \draw (-6,-2.5) -- (6,-2.5) node[pos=0.5,above] {\small{$\partial_y \mathfrak v= 0$}};


  \end{tikzpicture}
\end{equation}
Equation \eqref{weightedreadi} is a semilinear parabolic system, and thanks to the definition of $\rho$, has bounded coefficients. Moreover, the non-linearity $g(\mathfrak v):= \rho f\left(\frac{\mathfrak v}{\rho}\right)$ is Lipschitz with Lipschitz constant $\text{Lip} f$. Let

$$ C:= \text{Lip} f - \underset{\mathbb R}{\inf}\left( \frac{\rho''}{\rho} - 2\left( \frac{\rho'}{\rho} \right)^2 \right) > \text{Lip} f - \frac{d}{D} \underset{\mathbb R}{\inf}\left( \frac{\rho''}{\rho} - 2\left( \frac{\rho'}{\rho} \right)^2 \right)$$

Now define similarly $(\tilde{\mathfrak u},  \tilde{\mathfrak v}) = (\rho \tilde u, \rho \tilde v)$ and let $\mathfrak U:= e^{-Ct}\left(\mathfrak u - \tilde{\mathfrak u}\right), \mathfrak V:= e^{-Ct}\left(\mathfrak v - \tilde{\mathfrak v}\right)$. Observe that $(\mathfrak U, \mathfrak V)$ satisfies
\begin{equation}
\label{weightedreadidifference}
  \begin{tikzpicture}
  \draw (-7,0) -- (7,0) node[pos=0.5,below] {\small{$d\partial_y \mathfrak V = \mu \mathfrak U - \mathfrak V $}} node[pos=0.5,above] {$\partial_t \mathfrak U + 2\frac{\rho'}{\rho}\partial_x \mathfrak U - \partial_{xx}^2 \mathfrak U  = \mathfrak V -\mu \mathfrak U -  \left(\frac{\rho''}{\rho} - 2\left( \frac{\rho'}{\rho} \right)^2 + C \right) \mathfrak U$};

  \node at (0,-1.25) {$\partial_t \mathfrak V + 2\frac{d\rho'}{D\rho}\partial_x \mathfrak V - \frac{d}{D} \partial_{xx}^2 \mathfrak V - d\partial_{yy}^2 \mathfrak V  + \left(\frac{d}{D} \left( \frac{\rho''}{\rho} - 2\left( \frac{\rho'}{\rho} \right)^2 \right) + \frac{g(\mathfrak v)-g(\tilde{\mathfrak v})}{\mathfrak v - \tilde{\mathfrak v}} + C\right) \mathfrak V = 0$};

  \draw (-7,-3) -- (7,-3) node[pos=0.5,above] {\small{$\partial_y \mathfrak V= 0$}};


  \end{tikzpicture}
\end{equation}
By choice of $C$, the $0$-order terms in parentheses in equation \eqref{weightedreadidifference} are positive, thus equation \eqref{weightedreadidifference} enjoys the maximum principle and the maximum and minimum values of $(\mu \mathfrak U, \mathfrak V)$ are reached at initial time. Indeed, as usual if the maximum is reached by $\mathfrak V$, then either it is reached at initial time or it has to be reached on $y=0$ but there Hopf's lemma gives the contradiction $\mu \mathfrak U > \mathfrak V$. It $\mathfrak U$ reaches it, a contradiction is obtained at this point by seeing that the left-hand side in the equation satisfied by $\mathfrak U$ is non-negative: thus $\mathfrak V > \mu \mathfrak U$. In the end we have for all $ 0 \leq t < T$:
$$
\max(| (\mathfrak u - \tilde{\mathfrak u})(t) |_{L^\infty(\mathbb R)},| (\mathfrak v - \tilde{\mathfrak v})(t) |_{L^\infty(\Omega_L)})\leq e^{CT} \max\left( |(\mathfrak u - \tilde{\mathfrak u})(0) |_{L^\infty(\mathbb R)}, |(\mathfrak v - \tilde{\mathfrak v})(0) |_{L^\infty(\Omega_L)} \right),
$$
%
%
%
%
i.e. for all $t < T$, $x\in \mathbb R$, $y\in [-L,0]$:
$$ \rho(x) \left( |u-\tilde u|(t,x) + |v-\tilde v|(t,x,y) \right) \leq 2 e^{CT}  \| (u_0 - \tilde u_0, v_0 - \tilde v_0)\|_X  $$
and Lemma \ref{lemweight} follows by taking $C_{T,M} = 2e^{CT} \sup_{x\in (-M,M)} \frac{1}{\rho(x)}$, which is $2 e^{CT}e^M$ when $M$ is large. Observe that $C_{T,M}$ depends only on $T$, $M$ and $\text{Lip} f$. \end{proof}

\section*{Acknowledgement}
The research leading to these results has received funding from the European Research Council under the European Union's Seventh Framework Programme (FP/2007-2013) / ERC Grant Agreement n.321186 - ReaDi -Reaction-Diffusion Equations, Propagation and Modelling. Both authors were also partially funded by the ANR project NONLOCAL ANR-14-CE25-0013.

\eject
\section*{Appendix:  the heat kernel} 

We compute here the large time asymptotics of the solution $(u(t,x),v(t,x,y))$ to
\beq
\label{eA.1}
\begin{array}{rll}
u_t-u_{xx}+\mu u-v(t,x,0)=&0\  \  &(t>0,x\in\RR)\\
v_t-\e^2v_{xx}-v_{yy}=&0\  \  &(t>0,x\in\RR,y\in(-L,0))
\ \\
v_y(t,x,0)=&\mu u(t,x)-v(t,x,0)\   \ &(t>0,x\in\RR)
\end{array}
\eeq
with initial datum
\beq
\label{eA.2}
(u(0,x),v(0,x,y))=(u_0(x),0).
\eeq
Notice that we have, without loss of generality, set $d=1$. We limit ourselves to $u(t,x)$, as this is 
the quantity that will be useful to us. the method that we use is quite standard, a computationally much more involved case being treated in \cite{AnneCha} where the diffusion on the road is represented by the fractional Laplacian.

\noindent{\bf Proposition A.1.} {\it Set $a=\di\frac{1+\mu\e^2}{1+\mu}$. There are two constants $C>0$ and $\omega>0$, and a function $\delta(t)$ tending to 0 as $t\to+\infty$ such that we have,
for $t\geq 1$
$$
\biggl\vert u(t,x)-(1+\delta(t))\int_{\RR}\frac{e^{-(x-x')^2/4at}}{\sqrt{2\pi at}}u_0(x')dx'\biggl\vert\leq Ce^{-\omega t}\Vert u_0\Vert_{L^1(\RR)}.
$$}
\noindent{\sc Proof.} Let $(\hat u(t,\xi),\hat v(t,\xi,y))$ be the Fourier transform in $x$ of $(u,v)$; we have
\beq
\label{eA.3}
\begin{array}{rll}
\hat u_t+\xi^2 \hat u+\mu \hat u-\hat v(t,x,0)=&0\  \  &(t>0,\xi\in\RR)\\
\hat v_t-\e^2\hat v-\hat v_{yy}=&0\  \  &(t>0,\xi\in\RR,y\in(-L,0))
\ \\
\hat v_y(t,\xi,0)=&\mu u(t,\xi)-v(t,\xi,0)\   \ &(t>0,\xi\in\RR)
\end{array}
\eeq
with initial datum
$(\hat  u(0,\xi),\hat v(0,\xi,y))=(\hat u_0(x),0).$

\noindent{\bf 1. The case of large $\vert \xi\vert$.} Let $A(\xi)$ be the operator acting on 
$C((-L,0))$ with domain the set of all functions $w(y)\in C^2((-L,0))$ such that 
$w_y(-L)=w_y(0)+w(0)=0,
$
and defined by $A(\xi)w=-w_{yy}+\e^2\xi^2w$. Its first eigenvalue is (as is given by a simple computation)
$\lambda_0+\e^2\xi^2$, with $\lambda_0$ being the first positive root of 
$$
\sqrt{\lambda L}{\mathrm{tan}}\sqrt{\lambda L}=1;
$$
an eigenfunction is $\cos(\sqrt{\lambda_0L}(y+L)$, that we may bound from below by a real number $\delta_0$.
Thus the solution $w(t,\xi,y)$ of
$$
w_t-w_{yy}+\e^2\xi^2w=0,\ (t>0,-L<y<0),\   \    \  w_y(t,\xi,-L)=w_y(t,\xi,0)+w(t,\xi,0)=0,
$$
with initial datum $w_0(\xi,y)$ satisfies 
\beq
\label{eA.4}
\vert w(t,\xi,y)\vert\leq \delta_0^{-1}\Vert w_0(\xi,.)\Vert_{L^\infty((-L,0))}.
\eeq
Now, multiply the equation for $\hat u$ by $\hat u/\vert\hat u\vert$, this yields
$$
\partial_t\vert\hat u\vert+(\xi^2+\mu)\vert\hat u\vert\leq\mu\vert\hat v(t,\xi,0)\vert.
$$
From \eqref{eA.4} and the Duhamel formula we obtain
$$
\partial_t\vert\hat u\vert+(\xi^2+\mu)\vert\hat u\vert\leq\mu\delta_0^{-1}\int_0^te^{-\lambda_0(t-s)}\vert\hat u(s,\xi)\vert ds.
$$
Choose any $\omega_0\in(0,\lambda_0)$, we will prove the inequality
\beq
\label{eA.5}
\vert \hat u(t,\xi)\vert \leq \frac{C\vert \hat u_0(\xi)\vert e^{-\omega_0t}}{\xi^4-A}+\vert \hat u_0(\xi)\vert e^{-(\xi^2+\mu)t},
\eeq
for a universal $C$, a suitably chosen $A>0$ and $\vert\xi\vert\geq 10 A^{1/4}$. Set 
$U(t,\xi)=e^{\omega_0t}(\vert\hat u(t,\xi)\vert-\vert \hat u_0(\xi)\vert e^{-(\xi^2+\mu)t}),$
we have
\beq
\label{eA.6}
\partial_tU+(\xi^2+\mu-\omega_0)U\leq\mu\delta_0^{-1}\int_0^te^{-\lambda_0(t-s)}U(s,\xi)ds+\frac{\mu_0\delta^{-1}\vert \hat u_0(\xi)\vert e^{-(\lambda_0-\omega_0)t}}{\xi^2-(\lambda_0-\omega_0)},
\eeq
The function $t\mapsto U(t,\xi)$ is either decreasing or has positive maxima. In the first case, because $U(0,\xi)=0$ we have $\vert\hat u(t,\xi)\vert\leq \vert \hat u_0(\xi)\vert e^{-(\xi^2+\mu)t}$, so \eqref{eA.5} is proved. Let us assume the contrary, and let $t_0$ such that $U(t,\xi)\leq U(t_0,\xi)$ for all $t\leq t_0$. In such a case    we have,
from \eqref{eA.6}:
$$
(\xi^2+\mu-\frac{\mu\delta_0^{-1}}{\lambda_0-\omega_0})U(t_0,\xi)\leq \frac{\mu_0\delta^{-1}\vert \hat u_0(\xi)\vert e^{-(\lambda_0-\omega_0)t_0}}{\xi^2-(\lambda_0-\omega_0)}.
$$
So, \eqref{eA.5} is once again proved.

\noindent{\bf 2. The case of intermediate $\vert \xi\vert$.} We consider the range $\e_0\leq\vert\xi\vert\leq A$, $A$ being chosen so that \eqref{eA.5} holds. Let us this time consider the operator $L(\xi)$, acting on $\CC\times C((-L,0))$, its domain being all couples $(u,v)$ in $\CC\times C^2((-L,0))$ such that $w_y(-L)=0$ and $w_y(0)+w(0)=\mu u$, and its action being given by
\beq
\label{eA.7}
L(\xi)(u,w)=(\xi^2+\mu)u-w(0),-w_{ww}+\e^2\xi^2w).
\eeq
The family $L(\xi)$ is a family if sectorial operators with uniformly bounded coeffcients. Moreover, by Krein-Rutman's theorem, for $\vert \xi\vert \in\e_0,A)$, and $\e\in(0,1)$, $L(\xi)$ has a bottom eigenvalue $\lambda_0(\xi)$. Possibly, it depends on $\e$, but with a common positive lower bound depending on $\e_0$ that we call $\omega_0$ a positive  So, there is $\theta_0\in(0\pi/2)$ such that the path $\gamma=\omega_0+\RR e^{i\theta_0}$ encloses $\sigma(L(\xi))$, for all $\xi$ in the range that we consider. We have 
\beq
\label{eA.10}
e^{-tL(\xi)}(\hat u_0,\hat v_0)=\di\frac1{2i\pi}\int_\gamma e^{-\lambda t}(\lambda I-L(\xi))^{-1}(\hat u_0,0)d\lambda,
\eeq
an expression that admits a bound of the form 
\beq
\label{eA.8}
\Vert e^{-tL(\xi)}(\hat u_0(\xi),0)\Vert_{L^\infty((-L,0))}\leq Ce^{-\omega_0t}\vert\hat u_0(\xi)\vert.
\eeq

\noindent{\bf 3. The case of small $\vert \xi\vert$.} Fix $\e_0>0$ small so that the following (finite set of) considerations are true. Here we simply perform a Laplace transform of 
\eqref{eA.3}, we still call $\hat u(\lambda,\xi)$ the unknown in the new variables $(\lambda,\xi)$ this leads (after some standard algebra) to the 
system
\beq
\label{eA.9}
\biggl(\xi^2-\lambda-\frac{\mu\sqrt{\lambda-\e^2\xi^2}{\mathrm{tan}}(L\sqrt{\lambda-\e^2\xi^2)}}{1-\sqrt{\lambda-\e^2\xi^2}{\mathrm{tan}}(L\sqrt{\lambda-\e^2\xi^2)}}\biggl)\hat u(\lambda,\xi)=\hat u_0(\xi).
\eeq
The factor before $\hat u(\lambda,\xi)$ vanishes, for $\vert\xi\vert\leq\e_0$, at a unique $\lambda(\xi)$saisfying:
\beq
\label{eA.12}
\lambda(\xi)\sim\frac{(1+\mu\e^2)\xi^2}{1+\mu},
\eeq
 the next zeroes being further away, uniformly in $\xi$, by the principle of isolated zeroes. Then, \eqref{eA.10} remains valid and a standard inverse Laplace transform computation yields the existence of a function $\delta(t)=o_{t\to+\infty}(1)$ such that
\beq
\label{eA.11}
\hat u(t,\xi)-(1+\delta(t))e^{-t\lambda(\xi)}\hat u_0=O(e^{-\omega_0t})\vert \hat u_0\vert.
\eeq
The proposition is proved by taking the inverse Fourier transform of $\hat u$, and putting together estimates \eqref{eA.5}, \eqref{eA.8}, \eqref{eA.12} and \eqref{eA.11}. \hfill $\bullet$
\footnotesize{
\bibliographystyle{amsplain}
\bibliography{refs}
}

\end{document}